\documentclass[11pt, twoside]{article}
\pdfoutput=1

\usepackage[margin=1in]{geometry}

\usepackage{graphicx}
\usepackage{amsmath,amsfonts,amssymb}
\usepackage[mathscr]{eucal}
\usepackage{bm}

\usepackage{amsthm}
\usepackage{cite}
\usepackage{hyperref}

\usepackage{multirow}
\usepackage{makecell}
\usepackage{booktabs}

\usepackage{color}                              % For creating coloured text and background

\usepackage[nameinlink]{cleveref}
\numberwithin{equation}{section}

\newtheorem{theorem}{Theorem}[section]
\newtheorem{lemma}[theorem]{Lemma}

\newtheorem{corollary}[theorem]{Corollary}

\theoremstyle{definition}

\newtheorem{assumption}[theorem]{Assumption}
\newtheorem{example}[theorem]{Example}

\theoremstyle{remark}
\newtheorem{remark}[theorem]{Remark}

\usepackage{fancyhdr}
\begin{document}

\title {Superconvergent interpolatory HDG methods for reaction diffusion equations I: An HDG$_{k}$ method}
\author{Gang Chen%\fnref{myfootnote}}
	\thanks{School of Mathematics Sciences, University of Electronic Science and Technology of China, Chengdu, China, 56961491@qq.com}
	\and
	Bernardo Cockburn\thanks {School of Mathematics, University of Minnesota, Minneapolis, MN, cockburn@math.umn.edu}
	\and
	John~R.~Singler\thanks {Department of Mathematics and Statistics, Missouri University of Science and Technology, Rolla, MO, singlerj@mst.edu}
	\and
	Yangwen Zhang\thanks {Department of Mathematics and Statistics, Missouri University of Science and Technology, Rolla, MO, ywzfg4@mst.edu}
}

\date{}
\maketitle

	\begin{abstract}
		In our earlier work \cite{CockburnSinglerZhang1}, we approximated solutions of a general class of scalar parabolic semilinear PDEs by an interpolatory hybridizable discontinuous Galerkin (Interpolatory HDG) method.  This method reduces the computational cost compared to standard HDG since the HDG matrices are assembled once before the time integration.  Interpolatory HDG also achieves optimal convergence rates; however, we did not observe superconvergence after an element-by-element postprocessing.  In this work, we revisit the Interpolatory HDG method for reaction diffusion problems, and use the postprocessed approximate solution to evaluate the nonlinear term.  We prove this simple change restores the superconvergence and keeps the computational advantages of the Interpolatory HDG method.  We present numerical results to illustrate the convergence theory and the performance of the method.
		
	\end{abstract}
	
	\textbf{Keywords} Interpolatory hybridizable discontinuous Galerkin method, superconvergence

	\section{Introduction}
	In our earlier work \cite{CockburnSinglerZhang1}, we introduced an interpolatory hybridizable discontinuous Galerkin (Interpolatory HDG) method to approximate the solution of semilinear parabolic PDEs.  In contrast to standard HDG, the Interpolatory HDG method uses an elementwise interpolation procedure to approximate the nonlinear term; therefore, all quadrature for the nonlinear term can be performed once before the time integration, which yields a significant computational cost reduction.  The Interpolatory HDG method still converged at optimal rates, but superconvergence using element-by-element postprocessing was lost.  
	
	The superconvergence is an excellent feature of HDG methods, and therefore in this work we modify the Interpolatory HDG method from \cite{CockburnSinglerZhang1} and restore the superconvergence for reaction diffusion PDEs.
	
	Specifically, we consider the following class of scalar reaction diffusion PDEs on a Lipschitz polyhedral domain $\Omega\subset \mathbb R^d $, $ d = 2, 3$, with boundary $\partial\Omega$:
	\begin{equation}\label{semilinear_pde1}
	\begin{split}
	\partial_tu-\Delta u+ F(u)&= f \quad  \mbox{in} \; \Omega\times(0,T],\\
	u&=0 \quad  \mbox{on} \; \partial\Omega\times(0,T],\\
	u(\cdot,0)&=u_0~~\mbox{in}\ \Omega.
	\end{split}
	\end{equation}
	In \Cref{sec:HDG}, we provide background on HDG methods and describe the new Interpolatory HDG approach in detail.  We use the HDG$_k$ method to approximate the linear terms in the equation; i.e., $k$th order discontinuous polynomials are used to approximate the flux $\bm q = -\nabla u$, the scalar variable $u$, and its trace, and the stabilization function is chosen as $O(1)$ piecewise constant.  For the nonlinear term, we again use an elementwise Lagrange interpolation operator, as in \cite{CockburnSinglerZhang1}, but now we also approximate $ u $ using a postprocessing approach.  This modified approximate nonlinearity restores the superconvergence and, as in \cite{CockburnSinglerZhang1}, we have a simple explicit expressions for the nonlinear term and Jacobian matrix, which leads to an efficient and unified implementation.
	
	We analyze the semidiscrete Interpolatory HDG$_k$ method in \Cref{Error_analysis}.  We first assume the nonlinearity satisfies a global Lipschitz condition and prove the superconvergence.  Next, we establish the superconvergence under a local Lipschitz condition, assuming the mesh is quasi-uniform.
	
	In \Cref{sec:numerics}, we illustrate the convergence theory with numerical experiments and also demonstrate the performance of the Interpolatory HDG$_k$ method on a reaction diffusion PDE system.
	
	We note that interpolatory finite element methods for nonlinear PDEs are well-known to have computational advantages and have a long history.  The approach has been given many different names, including finite element methods with interpolated coefficients, product approximation, and the group finite element method.  For more information, see \cite{MR0502033,MR641309,MR798845,MR702221,MR967844,MR731213,MR2752869,MR2273051,MR2112661, MR1030644,MR1172090,MR973559,MR3178584,MR2294957,MR1068202,MR2391691,MR3403707,MR2587427} and the references therein.

\section{Interpolatory HDG$_k$ formulation and implementation}
\label{sec:HDG}
Hybridizable discontinuous Galerkin (HDG) methods were proposed by Cockburn et al.\ in \cite{MR2485455}.  HDG methods work with the mixed formulation of the PDE, and on each element the approximate solution and flux are expressed in terms of the approximate solution trace on the element boundary.  The approximate trace is uniquely determined by requiring the normal component of the numerical trace of the flux to be continuous across element boundaries.  This allows the approximate solution and approximate flux variables to be eliminated locally on each element; the result is a global system of equations for the approximate solution trace only.  Therefore, the number of globally coupled degrees of freedom for HDG methods is significantly lower than for standard DG methods.  HDG methods have been successfully applied to linear PDEs \cite{MR2485455,MR2772094,MR2629996,MR2513831} and nonlinear PDEs \cite{NguyenPeraireCockburnHDGAIAAINS10,PeraireNguyenCockburnHDGAIAACNS10,NguyenPeraireCM12,MoroNguyenPeraireSCL12,MR3626531,NguyenPeraireCockburnEDG15,KabariaLewCockburn15,MR2558780,MR3463051}. 

To describe the Interpolatory HDG$_k$ method, we introduce notation below.  We mostly follow the notation used in \cite{MR2485455}, where HDG methods were considered for linear, steady-state diffusion.

Let $\mathcal{T}_h$ be a collection of disjoint simplexes $K$ that partition $\Omega$.  Let $\partial \mathcal{T}_h$ denote the set $\{\partial K: K\in \mathcal{T}_h\}$. For an element $K$ in the collection  $\mathcal{T}_h$, let $e = \partial K \cap \Gamma$ denote the boundary face of $ K $ if the $d-1$ Lebesgue measure of $e$ is nonzero. For two elements $K^+$ and $K^-$ of the collection $\mathcal{T}_h$, let $e = \partial K^+ \cap \partial K^-$ denote the interior face between $K^+$ and $K^-$ if the $d-1$ Lebesgue measure of $e$ is nonzero. Let $\varepsilon_h^o$ and $\varepsilon_h^{\partial}$ denote the sets of interior and boundary faces, respectively, and let $\varepsilon_h$ denote the union of  $\varepsilon_h^o$ and $\varepsilon_h^{\partial}$. We use the mesh-dependent inner products
\begin{align*}
(w,v)_{\mathcal{T}_h} := \sum_{K\in\mathcal{T}_h} (w,v)_K,   \quad\quad\quad\quad\left\langle \zeta,\rho\right\rangle_{\partial\mathcal{T}_h} := \sum_{K\in\mathcal{T}_h} \left\langle \zeta,\rho\right\rangle_{\partial K},
\end{align*}
where $(\cdot,\cdot)_D$ denotes the $L^2(D)$ inner product for a set $D\subset\mathbb{R}^d$ and $\langle \cdot, \cdot\rangle_{\Gamma} $ denotes the $L^2(\Gamma)$ inner product for a set $\Gamma \subset \mathbb{R}^{d-1}$.

Let $\mathcal{P}^k(D)$ denote the set of polynomials of degree at most $k$ on a domain $D$.  We consider the discontinuous finite element spaces
\begin{align}
\bm{V}_h  &:= \{\bm{v}\in [L^2(\Omega)]^d: \bm{v}|_{K}\in [\mathcal{P}^k(K)]^d, \forall K\in \mathcal{T}_h\},\\
{W}_h  &:= \{{w}\in L^2(\Omega): {w}|_{K}\in \mathcal{P}^{k }(K), \forall K\in \mathcal{T}_h\},\\
{Z}_h  &:= \{{z}\in L^2(\Omega): {z}|_{K}\in \mathcal{P}^{k+1}(K), \forall K\in \mathcal{T}_h\},\\
{M}_h  &:= \{{\mu}\in L^2(\varepsilon_h): {\mu}|_{e}\in \mathcal{P}^k(e), \forall e\in \varepsilon_h,\mu|_{\varepsilon_h^\partial} = 0\}.
\end{align}
All spatial derivatives of functions in these spaces should be understood piecewise on each element $K\in \mathcal T_h$.

We consider the HDG method that approximates the scalar variable $ u $, flux $ \bm q = -\nabla u $, and boundary trace $ \widehat{u} $ using the spaces $ W_h $, $ \bm V_h $, and $ M_h $, respectively; i.e., polynomials of degree $ k $ are used for all variables.  We call this specific method HDG$_k$ to distinguish it from the wide variety of other available HDG methods, see, e.g., \cite{CockburnFuSayasM17, CockburnFuM2D17, CockburnFuM3D17, MR3507267}.  The space $ Z_h $ is used for postprocessing.

For the Interpolatory HDG$_k$ method, we use an elementwise interpolatory procedure along with postprocessing to approximate the nonlinear term.  Let $\mathcal I_h$ be the elementwise interpolation operator with respect to the finite element nodes for the postprocessing space $ Z_h $.  Therefore, for any function $ g $ that is continuous on each element we have $ \mathcal{I}_h g \in Z_h $.

The Interpolatory HDG$_k$ formulation reads: find $(\bm q_h,u_h,\widehat{u}_h)\in \bm V_h\times W_h\times M_h$ such that, for all $(\bm r_h,v_h,\widehat{v}_h)\in \bm V_h\times W_h\times M_h$, we have
\begin{subequations}\label{HDG-O}
	\begin{align}
	(\bm{q}_h,\bm{r}_h)_{\mathcal{T}_h}-(u_h,\nabla\cdot \bm{r}_h)_{\mathcal{T}_h}+\left\langle\widehat{u}_h,\bm r_h\cdot\bm n \right\rangle_{\partial{\mathcal{T}_h}} &= 0, \label{HDG-O_a}\\
	(\partial_t u_h, v_h)_{\mathcal T_h}-(\bm{q}_h,\nabla v_h)_{\mathcal{T}_h}+\left\langle\widehat{\bm{q}}_h\cdot \bm{n},v_h\right\rangle_{\partial{\mathcal{T}_h}} +  ( \mathcal I_h F(u_h^\star),v_h)_{\mathcal{T}_h}&= (f,v_h)_{\mathcal{T}_h},\label{HDG-O_b}\\
	\left\langle\widehat{\bm{q}}_h\cdot \bm{n}, \widehat{v}_h\right\rangle_{\partial{\mathcal{T}_h}\backslash\varepsilon^{\partial}_h} &=0,\label{HDG-O_c}\\
	u_h(0) &= \Pi u_0,\label{HDG-O_d}
	\end{align}
\end{subequations}
where $\Pi$ is a projection mapping into $ W_h $ and the numerical trace for the flux is defined by
\begin{align}\label{num_tra_HDG-O}
\widehat{\bm q}_h\cdot\bm n &= \bm q_h\cdot\bm n+\tau ( u_h -\widehat u_h).
\end{align}
Here, the stabilization function $ \tau $ is nonnegative, constant on each element, and $ O(1) $.  Furthermore, the postprocessed scalar variable $u_h^\star=\mathfrak{q}_h^{k+1}(\bm q_h, u_h) \in Z_h $ is determined on each element $ K $ by 
\begin{subequations}\label{post_process_1}
	\begin{align}
	(\nabla\mathfrak{q}_h^{k+1}(\bm q_h, u_h),\nabla z_h)_K&=-(\bm q_h,\nabla z_h)_K,\label{post_process_1_a}\\
	(\mathfrak{q}_h^{k+1}(\bm q_h, u_h),w_h)_K&=(u_h,w_h)_K,\label{post_process_1_b}
	\end{align}
\end{subequations}
 for all $(z_h, w_h)\in  [\mathcal P^{k+1}(K)]^{\perp}\times\mathcal{P}^{0}(K) $, where 
 \begin{align}
[\mathcal P^{k+1}(K)]^{\perp} = \{z\in \mathcal P^{k+1}(K) : (z,w_0)_K = 0 \ \textup{for all } w_0 \in \mathcal{P}^{0}(K)  \}.
 \end{align}
 
\begin{remark}
	In our original Interpolatory HDG work \cite{CockburnSinglerZhang1}, we used $ \mathcal{I}_h^k F(u_h) $ to approximate the nonlinear term, where $ \mathcal I_h^k $ is the elementwise interpolation operator mapping into $ W_h $.  We proved optimal convergence rates for all variables, but we did not observe superconvergence after an element-by-element postprocessing.  In this work, we approximate the nonlinearity using $ \mathcal I_h $ and postprocessing, i.e., $\mathcal I_h F(u_h^\star)$.  Note that this approximate nonlinearity is in $ Z_h $ instead of $ W_h $ as in our first work.  This simple change yields the superconvergence and keeps all the advantages of the original Interpolatory HDG$_k$ method proposed in \cite{CockburnSinglerZhang1}.  We provide details on the computational advantages of this approach in \Cref{sec:HDG2}.
\end{remark}

\subsection{ Implementation}
\label{sec:HDG2}
In our original work \cite{CockburnSinglerZhang1} on Interpolatory HDG, we provided details of the implementation for the method.  Since we changed the discretization of the nonlinear term in this work, the implementation is different; therefore, we provide details for the implementation of this new formulation and show how all matrices need only be assembled once before the time integration. As in our earlier work \cite{CockburnSinglerZhang1}, we describe the implementation using a simple time discretization approach: backward Euler with a Newton iteration to solve the nonlinear system at each time step.  Using Interpolatory HDG with other time discretization approaches is also possible.

Let  $N$ be a positive integer and define the time step $\Delta t = T/N$. We denote the approximation of $(\bm q_h(t),u_h(t),\widehat u_h(t))$ by $(\bm q^n_h,u^n_h,\widehat u^n_h)$ at the discrete time $t_n = n\Delta t $, for $n = 0,1,2,\ldots,N$. We replace the time derivative $\partial_tu_h$ in \eqref{HDG-O} by the backward Euler difference quotient 
\begin{align}
\partial^+_tu^n_h = \frac{u^n_h-u^{n-1}_h}{\Delta t}.\label{backward_Euler}
\end{align}
This gives the following fully discrete method: find $(\bm q^n_h,u^n_h,\widehat u^n_h)\in \bm V_h\times W_h\times M_h$ satisfying
\begin{subequations}\label{full_discretion_standard}
	\begin{align}
	(\bm{q}^n_h,\bm{r})_{\mathcal{T}_h}-(u^n_h,\nabla\cdot \bm{r})_{\mathcal{T}_h}+\left\langle\widehat{u}^n_h,\bm{r\cdot n} \right\rangle_{\partial{\mathcal{T}_h}} &= 0,\label{full_discretion_standard_a} \\
	(\partial^+_tu^n_h,w)_{\mathcal T_h}-(\bm{q}^n_h,\nabla w)_{\mathcal{T}_h}+\left\langle\widehat{\bm{q}}^n_h\cdot \bm{n},w\right\rangle_{\partial{\mathcal{T}_h}} + ( \mathcal I_hF( u_h^{n \star}),w)_{\mathcal{T}_h}&= (f^n,w)_{\mathcal{T}_h},\label{full_discretion_standard_b}\\
	\left\langle\widehat{\bm{q}}^n_h\cdot \bm{n}, \mu\right\rangle_{\partial{\mathcal{T}_h}\backslash\varepsilon^{\partial}_h} &=0,\label{full_discretion_standard_c}\\
	u^0_h &=\Pi  u_0,\label{full_discretion_standard_d}
	\end{align}
\end{subequations}
for all $(\bm r,w,\mu)\in \bm V_h\times W_h\times M_h$ and $n=1,2,\ldots,N$.  In \eqref{full_discretion_standard}, $f^n=f(t_n,\cdot)$, the numerical trace for the flux on $\partial\mathcal T_h$ is defined by
\begin{align}\label{num_tr_s}
\widehat{\bm q}_h^n\cdot\bm n=\bm q_h^n\cdot\bm n+\tau(u_h^n-\widehat u_h^n),
\end{align}	
and the postprocessed approximate solution $u_h^{n\star}$ is determined on each element $ K $ by solving
\begin{subequations}\label{post_process_1i}
	\begin{align}
	(\nabla u_h^{n\star},\nabla z_h)_K&=-(\bm q_h^n,\nabla z_h)_K,\label{post_process_1_ai}\\
	(u_h^{n\star},w_h)_K&=(u_h^n,w_h)_K,\label{post_process_1_bi}
	\end{align}
\end{subequations}
for all $(z_h, w_h)\in [\mathcal P^{k+1}(K)]^{\perp}\times\mathcal{P}^{0}(K) $. 

As is discussed below, the Interpolatory HDG$_k$ method takes great advantage of nodal basis functions; however, the postprocessing \eqref{post_process_1i} uses an orthogonal complement space, which complicates the implementation. To avoid this, on each element $ K $, we introduce a Lagrange multiplier $\eta_h^n\in \mathcal P^0(K)$ such that 
\begin{subequations}\label{post_process_2i}
	\begin{align}
	(\nabla u_h^{n\star},\nabla z_h)_K + (\eta_h^n,z_h)_K&=-(\bm q_h^n,\nabla z_h)_K,\label{post_process_2_ai}\\
	(u_h^{n\star},w_h)_K&=(u_h^n,w_h)_K,\label{post_process_2_bi}
	\end{align}
\end{subequations}
holds for all $(z_h, w_h)\in \mathcal P^{k+1}(K)\times\mathcal{P}^{0}(K) $. 
\begin{remark}
	In this work,  we used $w_h\in\mathcal P^{\ell}(K)$ with $\ell = 0$ in \eqref{post_process_2i}.  Actually, $\ell = 0,1, \ldots, k-1$ works both in the analysis and the numerical experiments.  In Part II of this work, we use the same postprocessing \eqref{post_process_2i} with $\ell = k$.
\end{remark}

Assume $\bm{V}_h = \mbox{span}\{\bm \varphi_i\}_{i=1}^{N_1}$, $W_h=\mbox{span}\{\phi_i\}_{i=1}^{N_2}$, $Z_h=\mbox{span}\{\chi_i\}_{i=1}^{N_3}$, and $M_h=\mbox{span}\{\psi_i\}_{i=1}^{N_4}$. Then
\begin{equation}\label{expre}
\begin{split}
\bm q^{n}_{h}= \sum_{j=1}^{N_1}\alpha_{j}^{n}\bm\varphi_j, \qquad 
u^{n}_h= \sum_{j=1}^{N_2} \beta_{j}^{n}\phi_j, \\
u^{n \star}_h= \sum_{j=1}^{N_3}\gamma_{j}^{n}\chi_j, \qquad 
\widehat{u}^{n}_h= \sum_{j=1}^{N_4}\zeta_{j}^{n}\psi_{j}.
\end{split}
\end{equation}
Also, define the following matrices
\begin{align*}
A_1 &= [(\nabla\chi_j,\nabla\chi_i )_{\mathcal{T}_h}], & A_2 &= [(\bm{\varphi}_j,\nabla\chi_i)_{\mathcal{T}_h}], & A_3 &= [(\bm\varphi_j,\bm\varphi_i )_{\mathcal{T}_h}],\\
A_4 &= [(\phi_j,\nabla\cdot\bm{\varphi}_i)_{\mathcal{T}_h}], & A_5 &= [\left\langle\psi_j,{\bm\varphi_i\cdot\bm n}\right\rangle_{\partial\mathcal{T}_h}], & A_6 &= [\left\langle\tau\phi_j,{\phi_i}\right\rangle_{\partial\mathcal{T}_h}],\\
A_7 &= [\left\langle\tau\psi_j,{\varphi_i}\right\rangle_{\partial\mathcal{T}_h}], & A_8 &= [\left\langle\tau\psi_j,{\psi_i}\right\rangle_{\partial\mathcal{T}_h}], & A_9 &= [(\chi_j,\phi_i)_{\mathcal{T}_h}],\\
  &  & M &= [(\phi_j,\phi_i )_{\mathcal{T}_h}], &  &
\end{align*}
and vectors
\begin{align*}
b_1 &= [(\chi_j,1 )_{\mathcal{T}_h}], &  b_2 &= [(\phi_j,1 )_{\mathcal{T}_h}], & b_3^n &= [(f^n,\phi_i )_{\mathcal{T}_h}].
\end{align*}
Since $ \bm V_h $, $ W_h $, and $ Z_h $ are discontinuous finite element spaces, many of the matrices are block diagonal with small blocks.

Substitute \eqref{expre} into the postprocessing equation \eqref{post_process_2i} and use the corresponding test functions to test \eqref{post_process_2i} on each element $K\in \mathcal T_h$. This gives the following local postprocessing equation
\begin{align*}
{\begin{bmatrix}
	A_1^k & (b_1^k)^T\\
	b_1^k & 0
	\end{bmatrix}}
{\left[ {\begin{array}{*{20}{c}}
		\bm\gamma^{n}_k\\
		\bm\eta^{n}_k\\
		\end{array}} \right]}
={\begin{bmatrix}
	-A_2^k & 0\\
	0& b_2^k
	\end{bmatrix}}
{\left[ {\begin{array}{*{20}{c}}
		\bm\alpha^{n}_k\\
		\bm\beta^{n}_k
		\end{array}} \right]},
\end{align*}
were $A_1^k$ is the $k$th block of the matrix $A_1$, and $A_2^k, b_1^k$, and $b_2^k$ are defined similarly.  That is,
\begin{align}
{\left[ {\begin{array}{*{20}{c}}
		\bm\gamma^{n}_k\\
		\bm\eta^{n}_k\\
		\end{array}} \right]}
&={\begin{bmatrix}
	A_1^k & (b_1^k)^T\\
	b_1^k & 0
	\end{bmatrix}}^{-1}
{\begin{bmatrix}
	-A_2^k & 0\\
	0& b_2^k
	\end{bmatrix}}
{\left[ {\begin{array}{*{20}{c}}
		\bm\alpha^{n}_k\\
		\bm\beta^{n}_k
		\end{array}} \right]}\nonumber\\
	& = {\begin{bmatrix}
		B_{11}^k & B_{12}^k\\
		B_{21}^k& B_{22}^k
		\end{bmatrix}}
	{\left[ {\begin{array}{*{20}{c}}
			\bm\alpha^{n}_k\\
			\bm\beta^{n}_k
			\end{array}} \right]},\label{post_process}
\end{align}
i.e.,
\begin{align*}
\bm \gamma^n_k = B_{11}^k \bm\alpha^{n}_k+  B_{12}^k \bm\beta^{n}_k.
\end{align*}
Let $B_{11}$ and $B_{12}$ be the block diagonal matrices with $k$th blocks $B_{11}^k$ and $B_{12}^k$, respectively.

As in \cite{CockburnSinglerZhang1}, once we test \eqref{full_discretion_standard_b} using $ w = \phi_i $ we can express the Interpolatory HDG$_k$ nonlinear term by the matrix-vector product
\begin{align*}
[ ( \mathcal I_h F(u_h^{n \star}),\phi_i)_{\mathcal{T}_h} ] &= A_9 \mathcal F(\bm\gamma^{n})\\
& = A_9 \mathcal F(B_{11} \bm\alpha^{n}+  B_{12} \bm\beta^{n}),
\end{align*}
where $\mathcal F$ is defined by
\begin{align}
\mathcal F(\bm\gamma^{n}) &= [F(\gamma_1^{n}), F(\gamma_2^{n}),\ldots,F(\gamma_{N_3}^{n})]^T.
\end{align}

Then the system \eqref{full_discretion_standard} can be rewritten as
\begin{align}\label{system_equation_group3}
\underbrace{\begin{bmatrix}
	A_3 & -A_4 &A_5\\
	A_4^T & \Delta t^{-1} M+A_6 &-A_7\\
	A_5^T & A_7^T &-A_8
	\end{bmatrix}}_{K}
\underbrace{\left[ {\begin{array}{*{20}{c}}
		\bm\alpha^{n}\\
		\bm\beta^{n}\\
		\bm\zeta^{n}\\
		\end{array}} \right]}_{\bm x_{n}}+
\underbrace{\left[ {\begin{array}{*{20}{c}}
		0\\
		A_9 \mathcal F(B_{11} \bm\alpha^{n}+  B_{12} \bm\beta^{n})\\
		0
		\end{array}} \right]}_{\mathscr F(\bm x_{n})}
=\underbrace{\left[ {\begin{array}{*{20}{c}}
		0\\
		b_3^n+{\Delta t}^{-1}M\bm\beta^{n-1} \\
		0
		\end{array}} \right]}_{\bm b_n},
\end{align}
i.e., 
\begin{align}\label{system_equation_group4}
K\bm x_n + \mathscr F(\bm x_n) = \bm b_n.
\end{align}

To apply Newton's method to solve the nonlinear equations \eqref{system_equation_group4}, define $G:\mathbb R^{N_1+N_2+N_4}\to \mathbb R^{N_1+N_2+N_4}$ by
\begin{align}\label{system_equation_group5}
G(\bm x_n ) = K\bm x_n + \mathscr F(\bm x_n) - \bm b_n.
\end{align}
At each time step $t_n $ for $ 1\le n\le N$, given an initial guess $\bm x_n^{(0)}$, Newton's method yields
\begin{align}\label{system_equation_group6}
\bm x_n^{(m)} =\bm x_n^{(m-1)} - \left[G'(\bm x_n^{(m-1)})\right]^{-1}G(\bm x_n^{(m-1)}),  \quad m=1,2,3,\ldots
\end{align}
where the Jacobian matrix $G'(\bm x_n^{(m-1)})$ is given by
\begin{align}\label{system_equation_group7}
G'(\bm x_n^{(m-1)}) = K+\mathscr F'(\bm x_n^{(m-1)}).
\end{align}
Similar to our earlier work \cite{CockburnSinglerZhang1} on Interpolatory HDG, the term $\mathscr F'(\bm x_n^{(m-1)})$ is easily computed by
\begin{align*}
\mathscr F'(\bm x_n^{(m-1)}) = \begin{bmatrix}
0 & 0 &0  \\
 A_{10}^{n,(m)}&A_{11}^{n,(m)}&0\\
0 &0  & 0 
\end{bmatrix},
\end{align*}
where $ A_{10}^{n,(m)} $ and $ A_{11}^{n,(m)} $ can be efficiently computed using sparse matrix operations by
\begin{align*}
A_{10}^{n,(m)} &= A_9 \, \text{diag}(\mathcal F'(B_{11} \bm\alpha^{n,(m-1)}+  B_{12} \bm\beta^{n,(m-1)})) B_{11},\\
A_{11}^{n,(m)} &= A_9 \, \text{diag}(\mathcal F'(B_{11} \bm\alpha^{n,(m-1)}+  B_{12} \bm\beta^{n,(m-1)})) B_{12}.
\end{align*}

Therefore, equation \eqref{system_equation_group6} can be rewritten as
\begin{align}\label{system_equation_group1}
\begin{bmatrix}
A_3 & -A_4 &A_5\\
A_4^T + A_{10}^{n,(m)} & \Delta t^{-1} M+A_6 + A_{11}^{n,(m)} &-A_7\\
A_5^T & A_7^T &-A_8
\end{bmatrix}
\left[ {\begin{array}{*{20}{c}}
	\bm\alpha^{n,(m)}\\
	\bm\beta^{n,(m)}\\
	\bm\zeta^{n,(m)}
	\end{array}} \right]
=\bm {\widetilde  b},
\end{align}
where 
\begin{align}
\bm {\widetilde  b} =  G'(\bm x_n^{(m-1)}) \bm x_n^{(m-1)} - G(\bm x_n^{(m-1)}).
\end{align}
This equation can be solved by locally eliminating the unknowns $\bm\alpha^{n,(m)}$ and $\bm\beta^{n,(m)}$; see \cite{CockburnSinglerZhang1} for details.

\begin{remark}
	In this new Interpolatory HDG$_k$ formulation, we only need to assemble the HDG matrices and the HDG postprocessing matrices $B_{11}$ and $B_{12}$ once before the time integration.  Hence, we keep all the advantages from our earlier work \cite{CockburnSinglerZhang1}: the new approach eliminates the computational cost of matrix reassembly and gives simple explicit expressions for the nonlinear term and Jacobian matrix, which leads to a simple unified implementation for a variety of nonlinear PDEs.
\end{remark}

\section{Error analysis}
\label{Error_analysis}

In this section, we give a rigorous error analysis for the semidiscrete Interpolatory HDG$_k$ method.  Below, we state our assumptions and briefly outline the main results.  Then we provide an overview of the projections required for the analysis in \Cref{subsec:basic_projections}.  The proofs of the main results follow.  We first assume in \Cref{subsec:analysis_global_Lip} that the nonlinearity satisfies a global Lipschitz condition.  Finally, in \Cref{local} we extend the results to locally Lipschitz nonlinearities; however, we assume the mesh is quasi-uniform and $ h $ is sufficiently small for this case.

We use the standard notation $W^{m,p}(\Omega)$ for Sobolev spaces on $\Omega$ with norm $\|\cdot\|_{m,p,\Omega}$ and seminorm $|\cdot|_{m,p,\Omega}$.  We also write $H^{m}(\Omega)$ instead of $W^{m,2}(\Omega)$, and we omit the index $p$ in the corresponding norms and seminorms.%  Also, we set $H_0^1(\Omega):=\{v\in H^1(\Omega):v=0 \;\mbox{on}\; \partial \Omega\}$.  Finally, we set $	H(\text{div},\Omega) := \{\bm{v}\in [L^2(\Omega)]^d, \nabla\cdot \bm{v}\in L^2(\Omega)\}.$

Throughout, we assume the solution of the PDE \eqref{semilinear_pde1} exists and is unique for $ t \in [0,T] $, the function $F$, the problem data, and the solution of the PDE are smooth enough, and the semidiscrete Interpolatory HDG$_k$ equations \eqref{HDG-O} have a unique solution on $ [0,T] $.  Furthermore, we assume the mesh is uniformly shape regular, $ h \leq 1 $, and the projection $ \Pi $ used for the initial condition in \eqref{HDG-O_d} is $ \Pi = \Pi_W $, where $ \Pi_W $ is defined below in \Cref{subsec:basic_projections}.

We also make the following regularity assumption on the dual problem: there exists a constant $ C $ such that for any $\Theta \in L^2(\Omega)$, the solution $ (\bm \Phi, \Psi) $ of the dual problem
\begin{equation}\label{Dual_PDE1_assumption}
\begin{split}
\bm{\Phi}+\nabla\Psi&=0\qquad\qquad\text{in}\ \Omega,\\
\nabla\cdot\bm \Phi   &=\Theta\qquad\quad~~\text{in}\ \Omega,\\
\Psi &= 0\qquad\qquad\text{on}\ \partial\Omega,
\end{split}
\end{equation}
satisfies $ (\bm \Phi, \Psi) \in [H^1(\Omega)]^d \times H^2(\Omega) $ and 
\begin{align}\label{regularity_PDE_assumption}
\|\bm \Phi\|_{H^{1}(\Omega)} + \|\Psi\|_{H^{2}(\Omega)} \le C   \|\Theta\|_{L^{2}(\Omega)}.
\end{align}
This assumption is satisfied if $ \Omega $ is convex.

We show for all $0\le t\le T$ the solution $ (\bm q_h, u_h, u_h^\star) $ of the semidiscrete Interpolatory HDG$_k$ equations \eqref{HDG-O} satisfies
\begin{align*}
\|\bm q(t) - \bm q_h(t)\|_{\mathcal T_h}&\le C h^{k+1}, \
\|u(t) - u_h(t)\|_{\mathcal T_h}\le C h^{k+1}, \ 
\|u(t) - u_h^\star(t)\|_{\mathcal T_h}\le C h^{k+1+\min\{k,1\}}.
\end{align*}
In our error estimates, the constants $ C $ can vary from line to line and may depend on the exact solution and the final time $ T $.  As in the linear case \cite{ChabaudCockburn12}, superconvergence is only obtained for $ k \geq 1 $.%  {\color{blue}Should we say something here about the $ \kappa $ term in the superconvergence result in \cite{ChabaudCockburn12}, and why we do not have that term?}

\begin{remark}
	In \cite{ChabaudCockburn12}, the $ L^\infty(L^2) $ error for $ u - u_h^* $ superconverges at a rate of $ \sqrt{ \log \kappa } \,\, h^{k+2} $, where $ \kappa $ depends on the mesh and the term $ \sqrt{ \log \kappa } $ grows very slowly as $ h $ tends to zero.  The term $ \sqrt{ \log \kappa } $ results from the parabolic duality argument used in \cite{ChabaudCockburn12}.  It appears this parabolic duality argument is not applicable to Interpolatory HDG.  Therefore, in this work we use a duality argument based on Wheeler's work \cite{MR0351124} and avoid the term $ \sqrt{ \log \kappa } $ in our error estimates; however, we require the solution has higher regularity than the regularity needed in \cite{ChabaudCockburn12} for the linear case.
\end{remark}

\subsection{Projections and basic estimates}
\label{subsec:basic_projections}

We first introduce the HDG$_k$ projection operator $\Pi_h(\bm{q},u) := (\bm{\Pi}_{V} \bm{q},\Pi_{W}u)$ defined in \cite{MR2629996}, where
$\bm{\Pi}_{V} \bm{q}$ and $\Pi_{W}u$ denote components of the projection of $\bm{q}$ and $u$ into $\bm{V}_h$ and $W_h$, respectively.
For each element $K\in\mathcal T_h$, the projection  is determined by the equations
\begin{subequations}\label{HDG_projection_operator}
	\begin{align}
	(\bm\Pi_V\bm q,\bm r)_K &= (\bm q,\bm r)_K,\qquad\qquad \forall \bm r\in[\mathcal P_{k-1}(K)]^d,\label{projection_operator_1}\\
	(\Pi_Wu, w)_K &= (u, w)_K,\qquad\qquad \forall  w\in \mathcal P_{k-1}(K	),\label{projection_operator_2}\\
	\langle\bm\Pi_V\bm q\cdot\bm n+\tau\Pi_Wu,\mu\rangle_{e} &= \langle\bm q\cdot\bm n+\tau u,\mu\rangle_{e},~\;\forall \mu\in \mathcal P_{k}(e),\label{projection_operator_3}
	\end{align}
\end{subequations}
for all faces $e$ of the simplex $K$.   The approximation properties of the HDG$_k$ projection \eqref{HDG_projection_operator} are given in the following result from \cite{MR2629996}:
\begin{lemma}\label{pro_error}
	Suppose $k\ge 0$, $\tau|_{\partial K}$ is nonnegative and $\tau_K^{\max}:=\max\tau|_{\partial K}>0$. Then the system \eqref{HDG_projection_operator} is uniquely solvable for $\bm{\Pi}_V\bm{q}$ and $\Pi_W u$. Furthermore, there is a constant $C$ independent of $K$ and $\tau$ such that
	\begin{subequations}
		\begin{align}
		\|{\bm{\Pi}_V}\bm{q}-q\|_K &\leq Ch_{K}^{\ell_{\bm{q}}+1}|\bm{q}|_{\ell_{\bm{q}}+1,K}+Ch_{K}^{\ell_{{u}}+1}\tau_{K}^{*}{|u|}_{\ell_{{u}}+1,K},\label{Proerr_q}\\
		\|{{\Pi}_W}{u}-u\|_K &\leq Ch_{K}^{\ell_{{u}}+1}|{u}|_{\ell_{{u}}+1,K}+C\frac{h_{K}^{\ell_{{\bm{q}}}+1}}{\tau_K^{\max}}{|\nabla\cdot \bm{q}|}_{\ell_{\bm{q}},K}\label{Proerr_u}
		\end{align}
	\end{subequations}
	for $\ell_{\bm{q}},\ell_{u}$ in $[0,k]$. Here $\tau_K^{*}:=\max\tau|_{{\partial K}\backslash e^{*}}$, where $e^{*}$ is a face of  $K$ at which $\tau|_{\partial K}$ is maximum.
\end{lemma}

Next, for each simplex $ K $ in $ \mathcal T_h $ and each boundary face $ e $ of $ K $, let $ \Pi_\ell $ (for any $ \ell \geq 0 $) and $ P_M $ denote the standard $L^2$ orthogonal projection operators $\Pi_{\ell}: L^2(K)\to \mathcal P^{\ell}(K)$ and $P_M : L^2(e)\to \mathcal P^{k}(e)$ satisfying
\begin{subequations}
	\begin{align}
	(\Pi_{\ell} u, v_h)_K &= (u,v_h)_K,\quad \forall v_h\in \mathcal P^{\ell}(K),\label{L2_do}\\
	\langle P_M  u, \widehat v_h \rangle_e &= \langle u, \widehat v_h \rangle_e,\quad \forall  \widehat v_h\in \mathcal P^{k}(e).\label{L2_edge}
	\end{align}
\end{subequations}
The following error estimates for the $ L^2 $ projections and the elementwise interpolation operator $ \mathcal I_h $ from \Cref{sec:HDG} are standard and can be found in \cite{MR2373954}:
\begin{lemma}\label{lemmainter}
	Suppose $k, \ell \ge 0$. There exists a constant $C$ independent of $K\in\mathcal T_h$ such that
	\begin{subequations}
		\begin{align}
		&\|w - \mathcal  I_h w\|_K + h_K\|\nabla(w - \mathcal  I_h w)\|_K \le Ch^{k+2} \|w\|_{k+2,K},  &  &\forall w\in C(\bar K)\cap H^{k+2}(K), \label{lemmainter_inter}\\
		&\|w - \Pi_{\ell}  w\|_K \le Ch^{\ell+1} \|w\|_{\ell+1,K},  &  &\forall w\in H^{\ell+1}(K), \label{lemmainter_orthoo}\\
		&\|w- P_M w\|_{\partial K} \le Ch^{k+1/2} \|w\|_{k+1,K}  &  &\forall w\in H^{k+1}(K). \label{lemmainter_orthoe}
		\end{align}
	\end{subequations}
\end{lemma}

\subsection{Error analysis under a global Lipschitz condition}
\label{subsec:analysis_global_Lip}
In this section, we assume the nonlinearity $F$ is globally Lipschitz:
\begin{assumption}\label{gloablly_lip}
 There is a constant $L>0$ such that 
\begin{align*}
|F(u) - F(v)|_{\mathbb R}\le L |u-v|_{\mathbb R}
\end{align*}
for all  $ u, v \in \mathbb{R}$. 
\end{assumption}
We remove this restriction in the next section. Our proof relies on techniques used in \cite{CockburnSinglerZhang1,ChabaudCockburn12}.  We split the proof of the main result into several steps.

To begin, we first rewrite the semidiscrete interpolatory HDG equations \eqref{HDG-O}.  First, subtract \eqref{HDG-O_c} from \eqref{HDG-O_b} and integrate by parts to give the following formulation:
\begin{lemma}
	The Interpolatory HDG$_k$ method finds $(\bm q_h,u_h,\widehat{u}_h)\in \bm V_h\times W_h\times M_h$ satisfying
	\begin{subequations}\label{HDGO}
		\begin{align}
		(\bm{q}_h,\bm{r}_h)_{\mathcal{T}_h}-(u_h,\nabla\cdot \bm{r}_h)_{\mathcal{T}_h}+\left\langle\widehat{u}_h,\bm r_h\cdot\bm n \right\rangle_{\partial{\mathcal{T}_h}} &= 0, \label{HDGO_a}\\
		(\partial_t u_h, v_h)_{\mathcal T_h} + ( \mathcal I_h F(u_h^\star),v_h)_{\mathcal{T}_h}+		(\nabla\cdot\bm{q}_h, v_h)_{\mathcal{T}_h}
		-\langle \bm q_h\cdot\bm n,\widehat{v}_h \rangle_{\partial{\mathcal{T}_h}}&
		\nonumber\\
		+\left\langle \tau( u_h -\widehat u_h),v_h-\widehat{v}_h)\right\rangle_{\partial{\mathcal{T}_h}} &= (f,v_h)_{\mathcal{T}_h},\label{HDGO_b}\\
		u_h(0) &= \Pi_W u_0,\label{HDGO_c}
		\end{align}
	\end{subequations}
	for all $(\bm r_h,v_h,\widehat{v}_h)\in \bm V_h\times W_h\times M_h$.
\end{lemma}
We also define the HDG$_k$ operator $\mathscr B$:
\begin{align}\label{def_B}
\begin{split}
\hspace{1em}&\hspace{-1em}\mathscr B (\bm q_h, u_h, \widehat u_h;\bm r_h, v_h, \widehat v_h )\\
& = (\bm{q}_h,\bm{r}_h)_{\mathcal{T}_h}-(u_h,\nabla\cdot \bm{r}_h)_{\mathcal{T}_h}+\left\langle\widehat{u}_h,\bm r_h\cdot\bm n \right\rangle_{\partial{\mathcal{T}_h}}\\
&\quad + (\nabla\cdot\bm{q}_h, v_h)_{\mathcal{T}_h}
-\langle \bm q_h\cdot\bm n,\widehat{v}_h \rangle_{\partial{\mathcal{T}_h}} +\left\langle
\tau( u_h -\widehat u_h),v_h-\widehat{v}_h)\right\rangle_{\partial{\mathcal{T}_h}}.
\end{split}
\end{align}
This allows us to rewrite the semidiscrete Interpolatory HDG$_k$ formulation \eqref{HDGO} as follows: find $(\bm q_h, u_h, \widehat u_h)\in \bm V_h\times W_h\times M_h$ such that
\begin{subequations}\label{HDGO-B}
\begin{align}
(\partial_t u_h, v_h)_{\mathcal T_h}+\mathscr B (\bm q_h, u_h, \widehat u_h;\bm r_h, v_h, \widehat v_h ) + ( \mathcal I_h F(u_h^\star),v_h)_{\mathcal{T}_h} &= (f,v_h),\\
u_h(0) &= \Pi_W u_0.
\end{align}
\end{subequations}
for all $(\bm r_h,v_h,\widehat{v}_h)\in \bm V_h\times W_h\times M_h$.% and $u_h(0)= \Pi_W u_0$.

\subsubsection{Step 1: Equations for the projection of the errors}
\begin{lemma}\label{error_u3}
	For $\varepsilon_h^{\bm q}=\bm{\Pi}_{V}\bm q-\bm q_h $, $ \varepsilon_h^{ u}=\Pi_{W} u-u_h $, and $ \varepsilon_h^{ \widehat{u}}=P_M u-\widehat{u}_h$, we have
	\begin{subequations}\label{error_u1}
		\begin{align}
		\hspace{3em}&\hspace{-3em}(\partial_t \varepsilon_h^u, v_h)_{\mathcal T_h } + \mathscr B(\varepsilon_h^{\bm q},\varepsilon_h^{u},\varepsilon_h^{\widehat u}; \bm r_h, w_h, \widehat v_h) + (F(u) - \mathcal I_h F(u_h^\star), v_h )_{\mathcal T_h} \nonumber\\
		& = (\bm \Pi_V \bm{q} - \bm q,\bm{r}_h)_{\mathcal{T}_h} + (\Pi_W u_t - u_t , v_h)_{\mathcal{T}_h},\label{error_u1_a}\\
		\varepsilon_h^u|_{t=0} &=0,\label{error_u1_b}
		\end{align}
	\end{subequations}
	for all $(\bm r_h,v_h,\widehat v_h)\in \bm V_h\times W_h\times M_h$.
\end{lemma}

\begin{proof}
	By the definition of the operator $\mathscr B$ in \eqref{def_B}, we have
	\begin{align*}
	\hspace{2em}&\hspace{-2em}  \mathscr B (\bm\Pi_V\bm q,\Pi_W u,P_M u,\bm r_h,v_h,\widehat v_h)\\
	&= (\bm\Pi_V\bm q,\bm{r}_h)_{\mathcal{T}_h}-(\Pi_W u,\nabla\cdot\bm{r}_h)_{\mathcal{T}_h}+\left\langle P_M u, \bm{r}_h\cdot \bm{n}\right\rangle_{\partial\mathcal{T}_h} +(\nabla\cdot \bm\Pi_V\bm q, v_h)_{\mathcal{T}_h}\\
	&\quad- \left\langle \bm\Pi_V\bm q \cdot\bm n , \widehat v_h\right\rangle_{\partial {\mathcal{T}_h}} +\left\langle \tau(\Pi_W u - u), v_h- \ \widehat v_h\right\rangle_{\partial\mathcal{T}_h}\\
	&= (\bm{\Pi}_V \bm q - \bm q,\bm{r}_h)_{\mathcal{T}_h}+ (\bm q,\bm{r}_h)_{\mathcal{T}_h}-( u,\nabla\cdot\bm{r}_h)_{\mathcal{T}_h}+\left\langle u, \bm{r}_h\cdot \bm{n}\right\rangle_{\partial\mathcal{T}_h} \\
	&\quad +(\nabla\cdot\bm q, v_h)_{\mathcal T_h} +(\nabla\cdot(\bm \Pi_V \bm q - \bm q), v_h)_{\mathcal T_h}- \left\langle (\bm\Pi_V\bm q - \bm q) \cdot\bm n , \widehat v_h\right\rangle_{\partial {\mathcal{T}_h}}\\
	&\quad +\left\langle \tau(\Pi_W u -  u), v_h- \widehat v_h\right\rangle_{\partial\mathcal{T}_h}\\
	&= (\bm{\Pi}_V \bm q - \bm q,\bm{r}_h)_{\mathcal{T}_h}+(f -F(u)+\partial_t u, v_h)_{\mathcal T_h},
	\end{align*}
	where we used the HDG$_k$ projection \eqref{HDG_projection_operator} and the $ L^2 $ projection $P_M$ \eqref{L2_edge}.  Use \eqref{HDGO-B} and subtract to obtain the result.
%	
%	Subtract the above from \eqref{HDGO-B} we get
%	\begin{align*}
%	(\partial_t \varepsilon_h^u, v_h)_{\mathcal T_h } + \mathscr B (\varepsilon_h^{\bm q},\varepsilon_h^{u},\varepsilon_h^{\widehat u}; \bm r_h, v_h, \widehat v_h) &=   (\bm \Pi_V \bm{q} - \bm q,\bm{r}_h)_{\mathcal{T}_h} + (\Pi_W u_t - u_t , v_h)_{\mathcal{T}_h}.
%	\end{align*}
\end{proof}

\subsubsection{Step 2: Estimate of $\varepsilon_h^u$ in $L^{\infty}(L^2)$ by an energy argument}
\label{sec:energy_argument_q2}
\begin{lemma}\label{super_con}
	For any $t\in[0,T]$, we have
	\begin{align*}
	\|\Pi_{k+1} u - u_h^\star\|_{\mathcal T_h}  &\le C ( \|\varepsilon_h^u\|_{\mathcal T_h} + \| u - \mathcal I_h u \|_{K} + \delta_{k0}\|\Pi_W u - u\|_{\mathcal T_h}) \\
	& \quad +Ch(\|\varepsilon_h^{\bm q}\|_{\mathcal T_h}+\|\bm{q}-\bm\Pi_V\bm{q}\|_{\mathcal T_h} +\|\nabla(u - \mathcal I_h u)\|_{\mathcal T_h}),
	\end{align*}
	where $ \delta_{k0} $ denotes the Kronecker delta symbol so that $ \delta_{k0} = 1 $ for $ k = 0 $ and $ \delta_{k0} = 0 $ for $ k \geq 1 $.
\end{lemma}
\begin{proof}
	We begin the proof with the case $k\ge 1$.  The proof is very similar to a proof in \cite{MR1086845}, but we include it for completeness.  By the properties of $\Pi_W$ and $\Pi_{k+1}$, we obtain
	\begin{align*}
	(\Pi_W u,w_0)_K &=(u,w_0)_K,\quad \text{ for all } w_0\in \mathcal{P}^{0}(K),\\
	(\Pi_{k+1} u,w_0)_K&=(u,w_0)_K, \quad \text{ for all } w_0\in \mathcal{P}^{0}(K).
	\end{align*}
	Hence, for all $w_{0}\in \mathcal P^{0}(K)$, we have
	\begin{align*}
	(\Pi_W u-\Pi_{k+1} u, w_{0})_K = 0.
	\end{align*}
	Let $e_h=u_h^\star - u_h+\Pi_W u-\Pi_{k+1} u$.  Using the postprocessing equation \eqref{post_process_1}, $ \bm q = -\nabla u $, and an inverse inequality gives
	\begin{align}
	\|\nabla e_{h}\|_K^2&=(\nabla (u_h^\star - u_h),\nabla e_{h} )_K+( \nabla (\Pi_Wu -\Pi_{k+1} u),\nabla e_{h} )_K \nonumber\\
	&=(-\nabla u_h-\bm{q}_h,\nabla e_{h} )_K+(  \nabla (\Pi_W u-\Pi_{k+1} u),\nabla e_{h} )_K \nonumber\\
	&=(\nabla (\Pi_Wu - u_h)-(\bm{q}_h-\bm\Pi_V\bm{q}) + (\bm{q}-\bm\Pi_V\bm{q}) +\nabla (u-\Pi_{k+1} u),\nabla e_{h})_K \nonumber\\
	&\le C (h_K^{-1}\|\Pi_W u- u_h\|_K+\|\varepsilon_h^{\bm q}\|_K + \|\bm{q}-\bm\Pi_V\bm{q}\|_{K} +\|\nabla(u - \Pi_{k+1} u)\|_{K} )\|\nabla e_{h}\|_K.\label{H1}
	\end{align}
	Since $(e_h,1)_K=0$,  apply the Poincar\'{e} inequality and the above estimate \eqref{H1} to give
	\begin{align*}
	\|e_{h}\|_K&\le C h_K \|\nabla e_h\|_K \le C \|\varepsilon_h^u\|_K+Ch_K(\|\varepsilon_h^{\bm q}\|_K +\|\bm{q}-\bm\Pi_V\bm{q}\|_{K} +\|\nabla(u - \Pi_{k+1} u)\|_{K} ).
	\end{align*}
	Next, estimate the last term using an inverse inequality:
	\begin{align*}
	  h_K \|\nabla(u - \Pi_{k+1} u)\|_{K}  &\leq  h_K \|\nabla(u - \mathcal I_h u)\|_{K} + h_K \|\nabla(\mathcal I_h u - \Pi_{k+1} u)\|_{K}\\
	   	&\leq  h \|\nabla(u - \mathcal I_h u)\|_{K} + \| \mathcal I_h u - \Pi_{k+1} u\|_{K}\\
		&\leq  h \|\nabla(u - \mathcal I_h u)\|_{K} + \| u - \mathcal I_h u \|_{K}.	
	\end{align*}
	This implies
	\begin{align*}
	  \|e_{h} \|_{\mathcal T_h} \le C ( \|\varepsilon_h^u\|_{\mathcal T_h} + \| u - \mathcal I_h u \|_{\mathcal T_h} ) +Ch (\|\varepsilon_h^{\bm q}\|_{\mathcal T_h} +\|\bm{q}-\bm\Pi_V\bm{q}\|_{\mathcal T_h} + \|\nabla(u - \mathcal  I_h u)\|_{\mathcal T_h}).
	\end{align*}
	Hence, we have
	\begin{align*}
	  \|\Pi_{k+1} u - u_h^\star\|_{\mathcal T_h}&\le\|\Pi_{k+1} u -\Pi_W u- u_h^\star + u_h\|_{\mathcal T_h} + \|\Pi_W u-u_h\|_{\mathcal T_h} \nonumber\\
	  &  \le C ( \|\varepsilon_h^u\|_{\mathcal T_h} + \| u - \mathcal I_h u \|_{\mathcal T_h} ) +Ch (\|\varepsilon_h^{\bm q}\|_{\mathcal T_h} +\|\bm{q}-\bm\Pi_V\bm{q}\|_{\mathcal T_h} + \|\nabla(u - \mathcal I_h u)\|_{\mathcal T_h}).
	\end{align*}
	This completes the proof for the case $ k \geq 1 $.

%	This implies
%	\begin{align*}
%	\|e_{h} \|_{\mathcal T_h} \le C \|\varepsilon_h^u\|_{\mathcal T_h}+Ch(\|\varepsilon_h^{\bm q}\|_{\mathcal T_h} +\|\bm{q}-\bm\Pi_V\bm{q}\|_{\mathcal T_h} +\|\nabla(u - \Pi_{k+1} u)\|_{\mathcal T_h}).
%	\end{align*}
%	Hence, we have
%	\begin{align*}
%	\|\Pi_{k+1} u - u_h^\star\|_{\mathcal T_h}&\le\|\Pi_{k+1} u -\Pi_W u- u_h^\star + u_h\|_{\mathcal T_h} + \|\Pi_W u-u_h\|_{\mathcal T_h} \nonumber\\
%	%
%	&  \le C \|\varepsilon_h^u\|_{\mathcal T_h}+Ch(\|\varepsilon_h^{\bm q}\|_{\mathcal T_h}+\|\bm{q}-\bm\Pi_V\bm{q}\|_{\mathcal T_h} +\|\nabla(u - \Pi_{k+1} u)\|_{\mathcal T_h}).
%	\end{align*}
%	This completes the proof for the case $ k \geq 1 $.
	
	For the case $k=0$, we follow the above steps in the proof of the $ k \geq 1 $ case but we replace the projection $\Pi_W$ with $\Pi_k$ to obtain
	\begin{align*}
	\|\Pi_{k+1} u - u_h^\star\|_{\mathcal T_h}&\le\|\Pi_{k+1} u -\Pi_k u- u_h^\star + u_h\|_{\mathcal T_h} + \|\Pi_k u-u_h\|_{\mathcal T_h} \nonumber\\
	&  \le C ( \|\Pi_k u-u_h\|_{\mathcal T_h} + \| u - \mathcal I_h u\|_{\mathcal T_h} ) \\
	& \quad +Ch(\|\varepsilon_h^{\bm q}\|_{\mathcal T_h}+\|\bm{q}-\bm\Pi_V\bm{q}\|_{\mathcal T_h} +\|\nabla(u - \mathcal I_h u)\|_{\mathcal T_h})\\
	&  \le C ( \|\varepsilon_h^u\|_{\mathcal T_h} + \|\Pi_k u - u\|_{\mathcal T_h} + \|\Pi_W u - u\|_{\mathcal T_h} + \| u - \mathcal I_h u\|_{\mathcal T_h})\\
	&\quad +Ch(\|\varepsilon_h^{\bm q}\|_{\mathcal T_h}+\|\bm{q}-\bm\Pi_V\bm{q}\|_{\mathcal T_h} +\|\nabla(u - \mathcal I_h u)\|_{\mathcal T_h}).
	\end{align*}
	The optimality of the $ L^2 $ projection gives $ \|\Pi_k u - u\|_{\mathcal T_h} \leq \|\Pi_W u - u\|_{\mathcal T_h} $, and this completes the proof.
\end{proof}

To bound the error in the nonlinear term, we split $F(u)-\mathcal I_hF( u_h^\star)$ as
\begin{align*}
F( u)-\mathcal I_hF( u_h^\star)  &= F(u)- \mathcal I_h F(u)    + \mathcal I_h F(u)  -  \mathcal I_h F(\Pi_{k+1} u)  +  \mathcal I_h F(\Pi_{k+1} u)  -\mathcal I_hF(u_h^\star)\\
&=: R_1 + R_2 + R_3.
\end{align*}
A bound for the first term $R_1$ follows directly from the standard FE interpolation error estimate \eqref{lemmainter_inter} in \Cref{lemmainter} due to the smoothness assumption for the function $F$.  Error bounds for $R_2$ and $R_3$ are given in the following result:
	\begin{lemma}\label{non_est}
		We have 
		\begin{align*}
		\|\mathcal I_h F(u)  -  \mathcal I_h F(\Pi_{k+1} u) \|_{\mathcal T_h}&\le C \|  u- \mathcal I_h u\|_{\mathcal T_h},\\
		\| \mathcal I_h F(\Pi_{k+1} u)  -\mathcal I_hF(u_h^\star)\|_{\mathcal T_h}  &\le C\| \Pi_{k+1} u- u_h^\star\|_{\mathcal T_h}.
		\end{align*}
	\end{lemma}
The proofs of the estimates in \Cref{non_est} are similar to proofs in \cite{CockburnSinglerZhang1} and also use $ \| \Pi_{k+1} u- u\|_{\mathcal T_h} \leq \|  u- \mathcal I_h u\|_{\mathcal T_h} $.  We omit the details.
\begin{lemma}\label{theorem_err_u3}
	We have the estimate
	\begin{align*}
	\|\varepsilon_h^{u}(t)\|^2_{\mathcal{T}_h}+ \int_0^t  \left[\|\varepsilon_h^{\bm q}\|^2_{\mathcal{T}_h}+\langle \tau(\varepsilon_h^u -\varepsilon_h^{\widehat{u}}),  \varepsilon_h^u -\varepsilon_h^{\widehat{u}}\rangle_{\partial{\mathcal{T}_h}} \right]dt \le  C \int_0^t \bm {\mathscr H}^2,
	\end{align*}
	%where $\bm {\mathscr H}$ is defined in \eqref{def_H}.
	where
	\begin{align}\label{def_H}
	\begin{split}
	\bm {\mathscr H} &=  	\|\bm\Pi_V {\bm{q}} -\bm q\|_{\mathcal T_h} + \| \Pi_W u_t - u_t \|_{\mathcal{T}_h}  + \delta_{k0}\|\Pi_W u - u\|_{\mathcal T_h}\\
	&\quad + \| F(u)- \mathcal I_h F(u) \|_{\mathcal T_h} + \| u - \mathcal I_h u\|_{\mathcal T_h} + h \|\nabla (u - \mathcal I_h u)\|_{\mathcal T_h}.
	\end{split}
	\end{align}
\end{lemma}
\begin{proof}
	Take $(\bm r_h,v_h,\widehat{v}_h)=(\varepsilon_h^{\bm q},\varepsilon_h^{u},\varepsilon_h^{\widehat u})$ in the error equation \eqref{error_u1} to give
	\begin{align}\label{error_q1_at_0}
	\begin{split}
	\hspace{2em}&\hspace{-2em} \frac 1 2\frac{d}{dt}\|\varepsilon_h^{u}\|^2_{\mathcal{T}_h}+ \|\varepsilon_h^{\bm q}\|^2_{\mathcal{T}_h}+\langle \tau(\varepsilon_h^u -\varepsilon_h^{\widehat{u}}),  \varepsilon_h^u -\varepsilon_h^{\widehat{u}}\rangle_{\partial{\mathcal{T}_h}} \\
	& = (\bm\Pi_V {\bm{q}} -\bm q, \varepsilon_h^{\bm q})_{\mathcal{T}_h} + (\Pi_W u_t - u_t , \varepsilon_h^{u})_{\mathcal{T}_h}  - (F(u) - \mathcal I_h F(u_h^\star),\varepsilon_h^u)_{\mathcal T_h}.
	\end{split}
	\end{align}
	Apply the Cauchy-Schwarz inequality to each term of the right-hand side of the above identity and use $ h \leq 1 $, \Cref{super_con}, and \Cref{non_est} to get
	\begin{align*}
	 \frac{d}{dt}\|\varepsilon_h^{u}\|^2_{\mathcal{T}_h}+ \|\varepsilon_h^{\bm q}\|^2_{\mathcal{T}_h}+\langle \tau(\varepsilon_h^u -\varepsilon_h^{\widehat{u}}),  \varepsilon_h^u -\varepsilon_h^{\widehat{u}}\rangle_{\partial{\mathcal{T}_h}} \le C  \bm {\mathscr H}^2 + C \|\varepsilon_h^{u}\|^2_{\mathcal{T}_h}.
	 \end{align*}
	Gronwall's inequality, $ \varepsilon_h^u(0) = 0 $, and $ e^{Ct} \leq e^{CT} $ give the result.
%	 where
%	 \begin{align}\label{def_H}
%	 \begin{split}
%	\bm {\mathscr H} &=  	\|\bm\Pi_V {\bm{q}} -\bm q\|_{\mathcal T_h}  + h^{\min\{k,1\}}(\|\Pi_k u - u\|_{\mathcal T_h} + \|\Pi_W u - u\|_{\mathcal T_h})\\
%	&\quad + \| F(u)- \mathcal I_h F(u) \|_{\mathcal T_h} +  \|u - \Pi_{k+1} u\|_{\mathcal T_h} + h \|\nabla (u - \Pi_{k+1} u)\|_{\mathcal T_h} \\
%	%
%	&\quad+ \| u - \mathcal I_h u\|_{\mathcal T_h} + h^{\min\{k,1\}}(\|\Pi_k u - u\|_{\mathcal T_h} + \|\Pi_W u - u\|_{\mathcal T_h}).
%	\end{split}
%	\end{align}
%	Integrate from $0$ to $t$ to obtain
%	\begin{align*}
%	\|\varepsilon_h^{u}(t)\|^2_{\mathcal{T}_h}+ \int_0^t  \left[\|\varepsilon_h^{\bm q}\|^2_{\mathcal{T}_h}+\langle \tau(\varepsilon_h^u -\varepsilon_h^{\widehat{u}}),  \varepsilon_h^u -\varepsilon_h^{\widehat{u}}\rangle_{\partial{\mathcal{T}_h}} \right]dt \le C\int_0^t \bm {\mathscr H}^2 + C \int_0^t 	\|\varepsilon_h^{u}(t)\|^2_{\mathcal{T}_h}.
%	\end{align*}
%	Hence, the Gronwall's inequality gives
%	\begin{align*}
%	\|\varepsilon_h^{u}(t)\|^2_{\mathcal{T}_h}+ \int_0^t  \left[\|\varepsilon_h^{\bm q}\|^2_{\mathcal{T}_h}+\langle \tau(\varepsilon_h^u -\varepsilon_h^{\widehat{u}}),  \varepsilon_h^u -\varepsilon_h^{\widehat{u}}\rangle_{\partial{\mathcal{T}_h}} \right]dt \le  e^{Ct}\int_0^t \bm {\mathscr H}^2. 
%	\end{align*}
\end{proof}

\subsubsection{Step 3: Estimate of $\varepsilon_h^{q}$ in $L^{\infty}(L^2)$ by an energy argument}

\begin{lemma}\label{error_ana_q}
	We have
	\begin{align*}
	\|\varepsilon_h^{\bm q}(t)\|_{\mathcal T_h}^2 +\|\sqrt{\tau}(\varepsilon_h^u(t)-\varepsilon_h^{\widehat u}(t))\|_{\partial{\mathcal{T}_h}}^2\le  C \bigg(\|(\bm\Pi_V {\bm{q}} -\bm q) (0)\|_{\mathcal{T}_h}^2+ \int_0^t 	\bm {\mathscr G}^2 \bigg),
	\end{align*}
	%where $\bm {\mathscr G}$ is defined in \eqref{def_G}.
	where
	\begin{align}\label{def_G}
	\begin{split}
	\bm {\mathscr G} &=  \|\bm\Pi_V {\bm{q}} -\bm q\|_{\mathcal T_h} + \|\Pi_Wu_t-u_t\|_{\mathcal T_h} + \|\bm\Pi_V {\bm{q}}_t -\bm q_t\|_{\mathcal{T}_h} + \delta_{k0} \|\Pi_W u - u\|_{\mathcal T_h} \\
	&\quad  + \| F(u)- \mathcal I_h F(u) \|_{\mathcal T_h} +  \| u - \mathcal I_h u\|_{\mathcal T_h} +  h\|\nabla (u - \mathcal I_h u)\|_{\mathcal T_h}.
	\end{split}
	\end{align}
\end{lemma}
\begin{proof}
	Take $(\bm r_h, v_h, \widehat v_h) = (\bm r_h,0 , 0)$ in the error equation \eqref{error_u1} and differentiate the result with respect to time.  Also take $(\bm r_h, v_h, \widehat v_h) = (0, v_h, \widehat v_h)$ in \eqref{error_u1} to get
	\begin{subequations}\label{err_eq_2}
		\begin{align}
		(\partial_t\varepsilon_h^{\bm q},\bm r_h)_{\mathcal T_h}-(\partial_t\varepsilon_h^u,\nabla\cdot\bm r_h)_{\mathcal T_h}+\langle \partial_t\varepsilon_h^{\widehat u},\bm r_h\cdot\bm n\rangle_{\partial\mathcal T_h} &= (\bm\Pi_V\bm q_t-\bm q_t,\bm r_h)_{\mathcal T_h},\label{err_eq_a2}\\
		(\partial_t\varepsilon_h^u,v_h)_{\mathcal T_h}+(\nabla\cdot\varepsilon_h^{\bm q}, v_h)_{\mathcal{T}_h}-\langle {\varepsilon}_h^{{\bm q}}\cdot\bm n,\widehat v_h\rangle_{\partial{\mathcal{T}_h}} \quad &\nonumber\\
		+\langle \tau(\varepsilon_h^u- \varepsilon_h^{\widehat u}), v_h- \widehat v_h\rangle_{\partial\mathcal{T}_h}+ (F(u)-\mathcal I_hF(u_h^\star),v_h)_{\mathcal T_h} &=(\Pi_Wu_t-u_t,v_h)_{\mathcal T_h},\label{err_eq_b2}\\
		\varepsilon_h^u|_{t=0}&=0,\label{err_eq_c2}
		\end{align}
	\end{subequations}
	for all $(\bm r_h,w_h,\widehat v_h)\in \bm V_h\times W_h\times M_h$.
	
	Next,  take  $\bm r_h=\varepsilon_h^{\bm q} $ in \eqref{err_eq_a2}, $v_h=\partial_t\varepsilon_h^u$  in \eqref{err_eq_b2}, and $\widehat v_h=\partial_t\varepsilon_h^{\widehat u}$  in \eqref{err_eq_b2} \color{black} to obtain
	\begin{align*}
	\hspace{1em}&\hspace{-1em}\|\partial_t\varepsilon_h^u\|^2_{\mathcal T_h}+(\partial_t\varepsilon_h^{\bm q},\varepsilon_h^{\bm q})_{\mathcal T_h}+\langle \tau(\varepsilon_h^u-\varepsilon_h^{\widehat u}),\partial_t\varepsilon_h^u-\partial_t\varepsilon_h^{\widehat u}\rangle_{\partial\mathcal T_h} \\
	&=(\bm\Pi_V\bm q_t-\bm q_t,\varepsilon_h^{\bm q})_{\mathcal T_h}+(\Pi_Wu_t-u_t,\partial_t\varepsilon_h^u)_{\mathcal T_h}-  (F( u)-\mathcal I_hF(u_h^\star),\partial_t\varepsilon_h^u)_{\mathcal T_h}.
	\end{align*}
	Integrating in time gives
	\begin{align*}
	\hspace{1em}&\hspace{-1em}\frac 1 2 [\|\varepsilon_h^{\bm q}(t)\|_{\mathcal T_h}^2 +\|\sqrt{\tau}(\varepsilon_h^u(t)-\varepsilon_h^{\widehat u}(t))\|_{\partial{\mathcal{T}_h}}^2]+\int_0^t \|\partial_t\varepsilon_h^{u}\|_{\mathcal T_h}^2 \\
	&= \frac 12 [\|\varepsilon_h^{\bm q}(0)\|_{\mathcal T_h}^2 +\|\sqrt{\tau}(\varepsilon_h^u(0)-\varepsilon_h^{\widehat u}(0))\|_{\partial{\mathcal{T}_h}}^2 ] +\int_0^t  (\bm\Pi_V {\bm{q}_t} -\bm q_t, \varepsilon_h^{\bm q})_{\mathcal{T}_h}\\
	&\quad +\int_0^t (\Pi_Wu_t-u_t,\partial_t\varepsilon_h^u)_{\mathcal T_h}  - \int_0^t  (F(u) - \mathcal I_h F( u_h^\star),\partial_t\varepsilon_h^u)_{\mathcal T_h}.
	\end{align*}
	Use the Cauchy-Schwarz inequality, $ h \leq 1 $, \Cref{super_con}, and \Cref{non_est} to get
	\begin{align*}
	\hspace{1em}&\hspace{-1em} \|\varepsilon_h^{\bm q}(t)\|_{\mathcal T_h}^2 +\|\sqrt{\tau}(\varepsilon_h^u(t)-\varepsilon_h^{\widehat u}(t))\|_{\partial{\mathcal{T}_h}}^2 \\
	&\le   \|\varepsilon_h^{\bm q}(0)\|_{\mathcal T_h}^2 +\|\sqrt{\tau}(\varepsilon_h^u(0)-\varepsilon_h^{\widehat u}(0))\|_{\partial{\mathcal{T}_h}}^2 + C \int_0^t \left( \bm {\mathscr G}^2 + \|\varepsilon_h^u\|_{\mathcal T_h}^2  + \|\varepsilon_h^{\bm q}\|_{\mathcal T_h}^2 \right).
	\end{align*}
%	where
%		 \begin{align}\label{def_G}
%	\begin{split}
%	\bm {\mathscr G} &=  \|\bm\Pi_V {\bm{q}} -\bm q\|_{\mathcal T_h} + \|\Pi_Wu_t-u_t\|_{\mathcal T_h} + \|\bm\Pi_V {\bm{q}}_t -\bm q_t\|_{\mathcal{T}_h}  \\
%	%
%	&\quad  + \| F(u)- \mathcal I_h F(u) \|_{\mathcal T_h} +  \|u - \Pi_{k+1} u\|_{\mathcal T_h} + \| u - \mathcal I_h u\|_{\mathcal T_h} \\
%	%
%	&\quad+  h\|\nabla (u - \Pi_{k+1} u)\|_{\mathcal T_h} + h^{\min\{k,1\}}(\|\Pi_k u - u\|_{\mathcal T_h} + \|\Pi_W u - u\|_{\mathcal T_h}).
%	\end{split}
%	\end{align}
	Apply the integral Gronwall's inequality  to obtain
	\begin{align*}
	\hspace{2em}&\hspace{-2em}\|\varepsilon_h^{\bm q}(t)\|_{\mathcal T_h}^2 +\|\sqrt{\tau}(\varepsilon_h^u(t)-\varepsilon_h^{\widehat u}(t))\|_{\partial{\mathcal{T}_h}}^2\\
	&\le  C \left( \|\varepsilon_h^{\bm q}(0)\|_{\mathcal T_h}^2 +\|\sqrt{\tau}(\varepsilon_h^u(0)-\varepsilon_h^{\widehat u}(0))\|_{\partial{\mathcal{T}_h}}^2 + \int_0^t  	\bm {\mathscr G}^2  + \int_0^t \|\varepsilon_h^u\|_{\mathcal T_h}^2\right).
	\end{align*}

	Next, let $t=0$ in \eqref{error_q1_at_0} and use $\varepsilon_h^{u}(0) = 0$ to get
	\begin{align*}
	\|\varepsilon_h^{\bm q}(0)\|_{\mathcal T_h}^2 +\|\sqrt{\tau}(\varepsilon_h^u-\varepsilon_h^{\widehat u})(0)\|_{\partial{\mathcal{T}_h}}^2 &= ((\bm\Pi_V {\bm{q}} -\bm q)(0), \varepsilon_h^{\bm q}(0))_{\mathcal{T}_h}.
	\end{align*}
	Therefore,
	\begin{align*}
	\|\varepsilon_h^{\bm q}(0)\|_{\mathcal T_h}^2 +\|\sqrt{\tau}(\varepsilon_h^u-\varepsilon_h^{\widehat u})(0)\|_{\partial{\mathcal{T}_h}}^2 \le   \|(\bm\Pi_V {\bm{q}} -\bm q)(0)\|_{\mathcal{T}_h}^2,
	\end{align*}
	and the estimate for $\|\varepsilon_h^u\|^2_{\mathcal T_h}$ in \Cref{theorem_err_u3} completes the proof.
\end{proof}

\subsubsection{Step 4: Superconvergent estimate for $\varepsilon_h^{u}$  in $L^{\infty}(L^2)$ by a duality argument}
\label{subsec:superconv_duality_argument}
To get a superconvergent rate for $\|\varepsilon_h^{u}\|_{\mathcal T_h}$, we adopt a duality argument from Wheeler \cite{MR0351124}.  In that work, an elliptic projection is used and it commutes with the time derivative. It is easy to check that the operator $\Pi_W$ defined in \eqref{HDG_projection_operator} commutes with the time derivative, i.e, 
$\partial_t \Pi_W u =\Pi_W u_t $.

For any $t\in[0,T]$, let $(\overline{\bm q}_h,\overline{u}_h,\widehat{\overline u}_h)\in \bm V_h\times W_h\times M_h$ be the solution of the following steady state problem
\begin{align}\label{HDGO_a2}
\mathscr B (\overline{\bm q}_h,\overline{u}_h,\widehat{\overline u}_h, \bm r_h, v_h, \widehat v_h )= (f - \Pi_W u_t - F(u),v_h)_{\mathcal{T}_h},
\end{align}
for all $(\bm r_h,v_h,\widehat{v}_h)\in \bm V_h\times W_h\times M_h$.%  Taking the partial derivative of \eqref{HDGO_a2} with respect to $t$ shows $(\partial_t \overline{\bm q}_h,\partial_t\overline{u}_h,\partial_t\widehat{\overline u}_h)\in \bm V_h\times W_h\times M_h$ is the solution of
%\begin{align}\label{HDGO_3}
%%
%\mathscr B (\partial_t\overline{\bm q}_h,\partial_t\overline{u}_h,\partial_t\widehat{\overline u}_h, \bm r_h, v_h, \widehat v_h )&= (f_t - \Pi_W u_{tt} - F'(u)u_t,v_h)_{\mathcal{T}_h},
%\end{align}
%for all $(\bm r_h,v_h,\widehat{v}_h)\in \bm V_h\times W_h\times M_h$.

The following estimates are proved in \Cref{AppendixA}.
\begin{lemma}\label{error_dual}
	For any $t\in[0,T]$, we have
	\begin{subequations}
		\begin{align}
		\|\bm{\Pi}_{V}\bm q- \overline{\bm q}_h\|_{\mathcal{T}_h} &\le  C (\|\bm q - \bm\Pi_V \bm q\|_{\mathcal T_h} +\|u_t - \Pi_W u_t\|_{\mathcal T_h}),\label{error_dual_a}\\
		\|\Pi_{W} {u}-\overline u_h\|_{\mathcal{T}_h} &\le  Ch^{\min\{k,1\}} (\|\bm q - \bm\Pi_V \bm q\|_{\mathcal T_h} +\|u_t - \Pi_W u_t\|_{\mathcal T_h}),\label{error_dual_b}\\
		\|\partial_t(\Pi_W u - \overline{u}_h)\|_{\mathcal T_h}  &\le  Ch^{\min\{k,1\}} (\|\bm q_t - \bm\Pi_V \bm q_t\|_{\mathcal T_h} +\|u_{tt} - \Pi_W u_{tt}\|_{\mathcal T_h}).\label{error_dual_c}
		\end{align}
	\end{subequations}
\end{lemma}

\begin{lemma}\label{error_e}
	Let  $e_h^{\bm q}=\bm q_h -\overline{\bm q}_h $, $ e_h^{ u}= u_h - \overline{u}_h $, and $ e_h^{ \widehat{u}}=\widehat{u}_h - {\widehat{\overline u}}_h$.  Then for any $t\in[0,T]$ we have
	\begin{align*}
	\|e_h^u(t)\|_{\mathcal T_h}^2 \le  \|\Pi_Wu_0 - \overline u_h(0) \|_{\mathcal T_h}^2 +C \int_0^t \left(\|\partial_t (\Pi_W u - \overline{u}_h)\|_{\mathcal T_h}^2 +h^2\|(\bm\Pi_V {\bm{q}} -\bm q)(0)\|_{\mathcal{T}_h}^2 + \bm {\mathscr K}^2 \right),
	\end{align*}
	%where $\bm {\mathscr K}$ is defined in \eqref{def_K}.
	where 
	\begin{align}\label{def_K}
	\begin{split}
	\bm {\mathscr K} &=  h(\|\bm\Pi_V {\bm{q}} -\bm q\|_{\mathcal T_h} + \|\Pi_Wu_t-u_t\|_{\mathcal T_h} + \|\bm\Pi_V {\bm{q}}_t -\bm q_t\|_{\mathcal{T}_h})  \\
	  &\quad  + \| F(u)- \mathcal I_h F(u) \|_{\mathcal T_h} +  \| u - \mathcal I_h u\|_{\mathcal T_h} + h\|\nabla (u - \mathcal I_h u)\|_{\mathcal T_h} + \delta_{k0} \|\Pi_W u - u\|_{\mathcal T_h}.
	\end{split}
	\end{align}
\end{lemma}

\begin{proof}
	By the definition of the operator $\mathscr B$ in \eqref{def_B}, we have
	\begin{align*}
	\hspace{1em}&\hspace{-1em}(\partial_t e_h^u, v_h)_{\mathcal T_h} + \mathscr B (e_h^{\bm q}, e_h^{u}, e_h^{\widehat u}, \bm r_h, v_h, \widehat v_h )\\
	& = (\partial_t u_h, v_h)_{\mathcal T_h} + \mathscr B (\bm q_h, u_h,\widehat u_h, \bm r_h, v_h, \widehat v_h ) - (\partial_t \overline{u}_h, v_h)_{\mathcal T_h} - \mathscr B (\overline{\bm{q}}_h, \overline u_h,\widehat {\overline u}_h, \bm r_h, v_h, \widehat v_h )\\
	& =(f, v_h)_{\mathcal T_h} - (\mathcal I_h F(u_h^\star), v_h)_{\mathcal T_h} - (\partial_t \overline{u}_h, v_h)_{\mathcal T_h} - (f - \Pi_W u_t - F(u),v_h)_{\mathcal{T}_h}\\
	& =(\partial_t (\Pi_W u - \overline{u}_h), v_h)_{\mathcal T_h} + (F(u)-\mathcal I_h F(u_h^\star), v_h)_{\mathcal T_h}.
	\end{align*}
	Take $(\bm r_h, v_h,\widehat v_h) = (e_h^{\bm q}, e_h^{u}, e_h^{\widehat u})$ and use \Cref{super_con}, \Cref{non_est}, \Cref{error_ana_q}, the bound
	$$
	  \| \varepsilon_h^u \| = \| \Pi_W u - \overline{u}_h - e_h^u \| \leq \| \Pi_W u - \overline{u}_h \| + \| e_h^u \|,
	$$	
	and also \Cref{error_dual} to give
	\begin{align*}
	\hspace{1em}&\hspace{-1em}\frac 12 \frac{d}{dt} \|e_h^u\|_{\mathcal T_h}^2 + \|e_h^{\bm q}\|_{\mathcal T_h}^2 + \langle \tau(e_h^{u} - e_h^{\widehat u}), e_h^{u}- e_h^{\widehat u} \rangle_{\partial\mathcal T_h}\\
	&\le  C \left( \bm {\mathscr K}^2 + \|\partial_t (\Pi_W u - \overline{u}_h)\|_{\mathcal T_h}^2 +h^2\|(\bm\Pi_V {\bm{q}} -\bm q)(0)\|_{\mathcal{T}_h}^2   +C \|e_h^u\|_{\mathcal T_h}^2 \right).
	\end{align*}
%	where 
%\begin{align}\label{def_K}
%	\begin{split}
%	\bm {\mathscr K} &=  h(\|\bm\Pi_V {\bm{q}} -\bm q\|_{\mathcal T_h} + \|\Pi_Wu_t-u_t\|_{\mathcal T_h} + \|\bm\Pi_V {\bm{q}}_t -\bm q_t\|_{\mathcal{T}_h})  \\
%	%
%	&\quad  + \| F(u)- \mathcal I_h F(u) \|_{\mathcal T_h} +  \|u - \Pi_{k+1} u\|_{\mathcal T_h} + \| u - \mathcal I_h u\|_{\mathcal T_h} \\
%	%
%	&\quad+  h\|\nabla (u - \Pi_{k+1} u)\|_{\mathcal T_h} + h^{\min\{k,1\}}(\|\Pi_k u - u\|_{\mathcal T_h} + \|\Pi_W u - u\|_{\mathcal T_h}).
%	\end{split}
%\end{align}

	Integration from $0$ to $t$ gives
	\begin{align*}
	\hspace{1em}&\hspace{-1em}  \|e_h^u(t)\|_{\mathcal T_h}^2 +\int_0^t  \left[\|e_h^{\bm q}\|_{\mathcal T_h}^2 + \langle \tau(e_h^{u} - e_h^{\widehat u}), e_h^{u}- e_h^{\widehat u} \rangle_{\partial\mathcal T_h}\right] \\
	&\le  \|e_h^u(0)\|_{\mathcal T_h}^2 +C \int_0^t (\|\partial_t (\Pi_W u - \overline{u}_h)\|_{\mathcal T_h}^2  +h^2 \|(\bm\Pi_V {\bm{q}} -\bm q)(0)\|_{\mathcal{T}_h}^2 +  \bm {\mathscr K}^2) +C \int_0^t \|e_h^u\|_{\mathcal T_h}^2.
	\end{align*}
	By Gronwall's inequality and $e_h^u(0) = u_h(0) - \overline u_h(0) = \Pi_Wu_0 - \overline u_h(0)  $, we have
	\begin{align*}
	\|e_h^u(t)\|_{\mathcal T_h}^2 \le  \|\Pi_Wu_0 - \overline u_h(0) \|_{\mathcal T_h}^2 +C \int_0^t (\|\partial_t (\Pi_W u - \overline{u}_h)\|_{\mathcal T_h}^2 +h^2\|(\bm\Pi_V {\bm{q}} -\bm q)(0)\|_{\mathcal{T}_h}^2 + \bm {\mathscr K}^2).
	\end{align*}
\end{proof}

A combination of \Cref{error_dual}  and  \Cref{error_e} gives the following lemma:%and the main result:
\begin{lemma}\label{supper}
	For any $t\in[0,T]$, we have
	\begin{align*}
	\|\varepsilon_h^u(t)\|_{\mathcal T_h}\le   Ch^{\min\{k,1\}} (\|\bm q(0) - \bm\Pi_V \bm q(0)\|_{\mathcal T_h} +\|u_t(0) - \Pi_W u_t(0)\|_{\mathcal T_h}) + C\int_0^t \bm {\mathscr L},
 \end{align*}
where 
\begin{align}\label{def_L}
\begin{split}
\bm {\mathscr L} &=  h (\|\bm\Pi_V {\bm{q}} -\bm q\|_{\mathcal T_h} + \|\Pi_Wu_t-u_t\|_{\mathcal T_h} ) +  h^{\min\{k,1\}} ( \|u_{tt} - \Pi_W u_{tt}\|_{\mathcal T_h} + \|\bm\Pi_V {\bm{q}}_t -\bm q_t\|_{\mathcal{T}_h} ) \\
&\quad  + \| F(u)- \mathcal I_h F(u) \|_{\mathcal T_h} +  \| u - \mathcal I_h u\|_{\mathcal T_h} +  h\|\nabla (u - \Pi_{k+1} u)\|_{\mathcal T_h} + \delta_{k0} \|\Pi_W u - u\|_{\mathcal T_h}.
\end{split}
\end{align}
\end{lemma}

%{\color{blue}Should we at least briefly mention here where the estimate for $ u_h^\star $ comes from?}
Using $ u - u_h = (u-\Pi_W u) + \varepsilon_h^u $, $ \bm q - \bm q_h = (\bm q-\Pi_V \bm q) + \varepsilon_h^{\bm q} $, $ u - u_h^\star = (u-\Pi_{k+1} u) + (\Pi_{k+1}u-u_h^\star) $, and the triangle inequality gives the main result:
\begin{theorem}\label{main_err_qu}
	If the nonlinearity $F$ satisfies the global Lipschitz condition in \Cref{gloablly_lip} and the assumptions at the beginning of \Cref{Error_analysis} hold, then for all $0\le t\le T$ the solution $ (\bm q_h, u_h,u_h^\star) $ of the semidiscrete Interpolatory HDG$_k$ equations satisfy
	\begin{align*}
	\|\bm q(t) - \bm q_h(t)\|_{\mathcal T_h} &\le   \|\bm q(t) - \bm{\Pi}_V \bm q(t)\|_{\mathcal T_h} + C\|(\bm\Pi_V {\bm{q}} -\bm q)(0)\|_{\mathcal{T}_h} + C \int_0^t \bm{\mathscr G}, \\
	\|u(t) - u_h(t)\|_{\mathcal T_h} & \le \|u(t) -  \Pi_W  u(t)\|_{\mathcal T_h} +  C\int_0^t   \bm{\mathscr H},\\
	\|u(t) - u_h^\star(t)\|_{\mathcal T_h}&\le  Ch^{\min\{k,1\}} (\|\bm q(0) - \bm\Pi_V \bm q(0)\|_{\mathcal T_h} +\|u_t(0) - \Pi_W u_t(0)\|_{\mathcal T_h}) \\
	&\quad +  \|u(t) -  \Pi_{k+1}  u(t)\|_{\mathcal T_h} + C \int_0^t \bm {\mathscr L},
	\end{align*}
	where $\bm{\mathscr G}$, $\bm{\mathscr H}$ and $\bm{\mathscr L}$ are defined in \eqref{def_G}, \eqref{def_H} and \eqref{def_L}, respectively.
\end{theorem}
By  \Cref{pro_error},   \Cref{lemmainter}, and  \Cref{main_err_qu}, we obtain convergence rates for smooth solutions.

\begin{corollary}\label{res_coll}
	%	Suppose the solution of is smooth enough and the degree of polynomial is $k$,
	If, in addition, $ u $, $ \bm q $, and $ F(u) $ are sufficiently smooth for $ t \in [0,T] $, then for all $0\le t\le T$ the solution $ (\bm q_h, u_h, u_h^\star) $ of the semidiscrete Interpolatory HDG$_k$ equations satisfy
	\begin{align*}
	\|\bm q(t) - \bm q_h(t)\|_{\mathcal T_h}&\le C h^{k+1}, \
	\|u(t) - u_h(t)\|_{\mathcal T_h}\le C h^{k+1}, \ 
	\|u(t) - u_h^\star(t)\|_{\mathcal T_h}\le C h^{k+1+\min\{k,1\}}.
	\end{align*}
\end{corollary}

\subsection{Error analysis under a local Lipschitz condition}
\label{local}
In applications, the nonlinearity $F$ might not satisfy the  global Lipschitz condition, see \Cref{gloablly_lip}. Instead, let 
\begin{align}\label{Max_value}
M = \max\{|u(t,x)|, x\in \overline{\Omega}, t\in [0,T]\} + 1.
\end{align}
In this section, we assume the mesh is quasi-uniform, the polynomial degree satisfies $k\ge 1$, and the nonlinearity $F$ satisfies the following local Lipschitz condition: 
\begin{assumption}\label{Locally_lip}
	There is a constant $L(M)>0$ such that% for all $x\in \Omega$ and $u(t,x), v(t,x)\in [-M, M]$ such that 
	\begin{align*}
	|F(u) - F(v)|_{\mathbb R}\le L(M) |u-v|_{\mathbb R}
	\end{align*}
	for all  $ u, v \in [-M, M]$. 
\end{assumption}
Our proof relies on techniques used in \cite{MR3403707}.  Below, we use the notation $ \rho_h^{u^\star} = \Pi_{k+1} u - u_h^\star $.

\begin{lemma}\label{theorem_err_u3local}
	If $h$ is small enough and $ k \geq 1 $, then there exists $t^*_h \in (0, T]$ such that \Cref{error_ana_q} and \Cref{supper} hold for all $t\in [0, t^*_h]$.
%	\begin{subequations}\label{conclusion}
%	\begin{align}
%		\|\varepsilon_h^{u}(t)\|_{\mathcal{T}_h} &\le  C\int_0^t \|\bm\Pi_V {\bm{q}} -\bm q\|_{\mathcal T_h} + \| F(u)- \mathcal I_h F(u) \|_{\mathcal T_h} +  \|u - \Pi_{k+1} u\|_{\mathcal T_h},\\
%		%
%		\|\varepsilon_h^{\bm q}(t)\|_{\mathcal T_h} & \le  C\|\bm\Pi_V {\bm{q}} -\bm q)(0)\|_{\mathcal{T}_h} + C \int_0^t   \|F(u)- \mathcal I_h F(u) \|_{\mathcal T_h}+ \|\bm\Pi_V {\bm{q}}_t -\bm q_t\|_{\mathcal{T}_h} \nonumber\\
%		%
%		&\quad+C \int_0^t \|\Pi_Wu_t-u_t\|_{\mathcal T_h}+ \|\bm{\Pi}_V \bm q  - \bm q\|_{\mathcal T_h} + \|{\Pi}_W u  - u\|_{\mathcal T_h},\\
%		%
%		\|\rho_h^{u^\star}(t)\|_{\mathcal T_h}&\le    Ch^{\min\{k,1\}} (\|\bm q(0) - \bm\Pi_V \bm q(0)\|_{\mathcal T_h} +\|u_t(0) - \Pi_W u_t(0)\|_{\mathcal T_h})\\
%		%
%		%
%		%
%		&\quad+ C h^{\min\{k,1\}}\int_0^t \|\bm q - \bm\Pi_V \bm q\|_{\mathcal T_h}  + \|\bm q_t - \bm\Pi_V \bm q_t\|_{\mathcal T_h} +\|u_t - \Pi_W u_t\|_{\mathcal T_h} + \|u_{tt} - \Pi_W u_{tt}\|_{\mathcal T_h} \\
%		%
%		&\quad +C\int_0^t  (\|\Pi_{k+1} u - u\|_{\mathcal T_h} + h\|\nabla (u - \Pi_{k+1} u)\|_{\mathcal T_h} +  \|F(u) - \mathcal I_h F(u)\|_{\mathcal T_h}).
%	\end{align}
%\end{subequations}
\end{lemma}
\begin{proof}
	Take $(\bm r_h,v_h,\widehat{v}_h)=(\varepsilon_h^{\bm q},\varepsilon_h^{u},\varepsilon_h^{\widehat u})$ in \eqref{error_u1} to give
	\begin{align}\label{error_q1_at_00}
	\begin{split}
	\hspace{2em}&\hspace{-2em} \frac 1 2\frac{d}{dt}\|\varepsilon_h^{u}\|^2_{\mathcal{T}_h}+ \|\varepsilon_h^{\bm q}\|^2_{\mathcal{T}_h}+\langle \tau(\varepsilon_h^u -\varepsilon_h^{\widehat{u}}),  \varepsilon_h^u -\varepsilon_h^{\widehat{u}}\rangle_{\partial{\mathcal{T}_h}} \\
	& = (\bm\Pi_V {\bm{q}} -\bm q, \varepsilon_h^{\bm q})_{\mathcal{T}_h}   - (F(u) - \mathcal I_h F(u_h^\star),\varepsilon_h^u)_{\mathcal T_h}.
	\end{split}
	\end{align}
	Take $t=0$ and use the fact $\varepsilon_h^u(0)  = 0$ to obtain
	\begin{align*}
	\|\varepsilon_h^{\bm q}(0)\|_{\mathcal T_h} \le C \|\bm\Pi_V {\bm{q}}(0) -\bm q(0)\|_{\mathcal T_h}.
	\end{align*}
	By \Cref{super_con}, we have 
	\begin{align*}
	\|\rho_h^{u^\star}(0)\|_{\mathcal T_h}\le C \| u(0) - \mathcal I_h u(0) \|_{\mathcal T_h} + Ch(\|\bm\Pi_V {\bm{q}}(0) -\bm q(0)\|_{\mathcal T_h} + \|\nabla(u(0) - \mathcal I_h u(0))\|_{\mathcal T_h}).
	\end{align*}
	The inverse inequality gives
	\begin{align*}
	\|\rho_h^{u^\star}(0)\|_{L^\infty(\Omega)} &\le C h^{-\frac  d 2}\| u(0) - \mathcal I_h u(0) \|_{\mathcal T_h}\\
	&\quad + Ch^{1-\frac  d 2}(\|\bm\Pi_V {\bm{q}}(0) -\bm q(0)\|_{\mathcal T_h} + \|\nabla(u(0) - \mathcal I_h u(0))\|_{\mathcal T_h}).
	\end{align*}
	Since the exact solution is smooth at $ t = 0 $, we can choose $ h $ small enough so that $ \|\rho_h^{u^\star}(0)\|_{L^\infty(\Omega)} < 1/2 $.  Also, since the error equation \eqref{error_u1} is continuous with respect to the time $t$, again using an inverse inequality shows that there exists $t^*_h \in(0, T]$ such that for all $ h $ small enough,
	\begin{align}\label{proof_error_u1}
	\|\rho_h^{u^\star}(t)\|_{L^\infty(\Omega)} \le 1/2  \quad  \mbox{for all $t\in [0, t^*_h]$.}
	\end{align}
	For all $h$ sufficiently small we have
	\begin{align}\label{eqn:u_L2_proj_error_small_enough}
	\| u(t) - \Pi_{k+1} u(t) \|_{L^\infty(\Omega)} \le 1/2  \quad  \mbox{for all $t\in [0, t^*_h]$.}
	\end{align}
	This implies for all $ h $ small enough and all $t\in [0, t^*_h]$,
	\begin{align*}
	  \|  \Pi_{k+1} u \|_{L^\infty}  &\leq  \| u \|_{L^\infty} + \| u - \Pi_{k+1} u \|_{L^\infty}  \leq  \| u \|_{L^\infty} + 1/2  \leq  M,\\
	  \|  u_h^\star \|_{L^\infty}  &\leq  \| \Pi_{k+1} u \|_{L^\infty} + \|  \Pi_{k+1} u - u_h^\star \|_{L^\infty}  \leq  ( \| u \|_{L^\infty} + 1/2 ) + 1/2  =  M.
	\end{align*}
	
	Therefore, $u$, $ \Pi_{k+1} u $, and $u_h^\star$ are located in the interval $[-M,M]$, where $M$ is defined in \eqref{Max_value}. This allows us to take advantage of the local Lipschitz condition of $F(u)$ for all $t \in [0, t^*_h]$. Hence, \Cref{error_ana_q} and \Cref{supper} hold for all $t\in [0, t^*_h]$.
	%follow the proof of \Cref{theorem_err_u3}, we have 
%	\begin{align}
%	\|\varepsilon_h^{u}(t)\|_{\mathcal{T}_h} \le  C\int_0^t \|\bm\Pi_V {\bm{q}} -\bm q\|_{\mathcal T_h} + \| F(u)- \mathcal I_h F(u) \|_{\mathcal T_h} +  \|u - \Pi_{k+1} u\|_{\mathcal T_h}.
%	\end{align}
%	Follow the proof of \Cref{error_ana_q}, we have 
%	\begin{align*}
%	\hspace{1em}&\hspace{-1em} \|\varepsilon_h^{\bm q}(t)\|_{\mathcal T_h} +\|\sqrt{\tau}(\varepsilon_h^u(t)-\varepsilon_h^{\widehat u}(t))\|_{\partial{\mathcal{T}_h}}\\
%	%
%	& \le  C\|\bm\Pi_V {\bm{q}} -\bm q)(0)\|_{\mathcal{T}_h} + C \int_0^t   \|F(u)- \mathcal I_h F(u) \|_{\mathcal T_h}+ \|\bm\Pi_V {\bm{q}}_t -\bm q_t\|_{\mathcal{T}_h} \\
%	%
%	&\quad+C \int_0^t \|\Pi_Wu_t-u_t\|_{\mathcal T_h}+ \|\bm{\Pi}_V \bm q  - \bm q\|_{\mathcal T_h} + \|{\Pi}_W u  - u\|_{\mathcal T_h}.
%	\end{align*}
%	Follow the proof of \Cref{supper}, we have 
%		\begin{align*}
%	\|\varepsilon_h^u(t)\|_{\mathcal T_h}&\le   Ch^{\min\{k,1\}} (\|\bm q(0) - \bm\Pi_V \bm q(0)\|_{\mathcal T_h} +\|u_t(0) - \Pi_W u_t(0)\|_{\mathcal T_h})\\
%	%
%	%
%	%
%	&\quad+ C \int_0^t \left[h^{\min\{k,1\}} (\|\bm q - \bm\Pi_V \bm q\|_{\mathcal T_h} +\|u_t - \Pi_W u_t\|_{\mathcal T_h}) \right] \\
%	%
%	&\quad +C\int_0^t  (\|\Pi_{k+1} u - u\|_{\mathcal T_h} + h\|\nabla (u - \Pi_{k+1} u)\|_{\mathcal T_h} +  \|F(u) - \mathcal I_h F(u)\|_{\mathcal T_h}).
%	\end{align*}
\end{proof}

\begin{lemma}\label{theorem_err_u3_extend} 
	For $h$ small enough and $k\ge 1$, the conclusions of \Cref{error_ana_q} and \Cref{supper}  are true on the whole time interval $[0,T]$.
%	\begin{align}
%	\begin{split}
%	\hspace{2em}&\hspace{-2em}
%	\|\varepsilon_h^{u}(T)\|^2_{\mathcal{T}_h}+ \int_0^t  \left[\|\varepsilon_h^{\bm q}\|^2_{\mathcal{T}_h}+\langle \tau(\varepsilon_h^u -\varepsilon_h^{\widehat{u}}),  \varepsilon_h^u -\varepsilon_h^{\widehat{u}}\rangle_{\partial{\mathcal{T}_h}} \right]dt \\
%	%
%	& \le  C\int_0^t (\|\bm\Pi_V {\bm{q}} -\bm q\|_{\mathcal T_h}^2 + \| F(u)- \mathcal I_h F(u) \|_{\mathcal T_h}^2  +  \|u - \Pi_{k+1} u\|_{\mathcal T_h}^2) dt.
%	\end{split}
%	\end{align}
\end{lemma}
\begin{proof}
	Fix $ h^* > 0 $ so that \eqref{proof_error_u1}, \eqref{eqn:u_L2_proj_error_small_enough}, and \Cref{theorem_err_u3local} are true for all $ h \leq h^* $, and assume $t^*_h$ is the largest value for which \eqref{proof_error_u1} is true for all $ h \leq h^* $.  Define the set $ \mathcal{A} = \{ h \in [0,h^*] : t^*_h \neq T \} $.  If the result is not true, then $ \mathcal{A} $ is nonempty, $ \inf \{ h : h\in \mathcal{A} \} = 0 $, and also
	\begin{align}\label{assum}
	\|\rho_h^{u^\star}(t^*_h)\|_{L^\infty(\Omega)} =1/2  \quad  \mbox{for all $ h \in \mathcal{A} $.}
	\end{align}
	However, by the inverse inequality, $k\ge 1$, and since \Cref{theorem_err_u3local} is true, we have
	\begin{align*}
	\|\rho_h^{u^\star}(t^*_h)\|_{L^{\infty}(\Omega)}  \le Ch^{-\frac d 2}\|\rho_h^{u^\star}(t^*_h)\|_{\mathcal T_h}\le C h^{3-d/2}  \quad  \mbox{for all $ h \in \mathcal{A} $.}
	\end{align*}
	Since $ C $ does not depend on $ h $, there exists $ h^*_1 \leq h^* $ such that $ \|\rho_h^{u^\star}(t^*_h)\|_{L^{\infty}(\Omega)}<1/2 $ for all $ h \in \mathcal{A} $ such that $ h \leq h^*_1 $.  This contradicts \eqref{assum}, and therefore $t^*_h = T$ for all $ h $ small enough.
\end{proof}
%
%%
%\begin{proof}
%	Fix $ h^* > 0 $ so that \eqref{proof_error_u1} and \Cref{theorem_err_u3local} are true for all $ h \leq h^* $, and assume $t^*_h$ is the largest value for which \eqref{proof_error_u1} is true for all $ h \leq h^* $.  If there exists $ h^*_1 \leq h^* $ such that $t^*_h \neq T$ for all $ h \leq h^*_1 $, then
%	\begin{align}\label{assum}
%	\|\rho_h^{u^\star}(t^*_h)\|_{L^\infty(\Omega)} =1/2  \quad  \mbox{for all $ h \leq h^*_1 $.}
%	\end{align}
%	However, by the inverse inequality, $k\ge 1$, and since \Cref{theorem_err_u3local} is true, we have
%		\begin{align*}
%	\|\rho_h^{u^\star}(t^*_h)\|_{L^{\infty}(\Omega)}  \le Ch^{-\frac d 2}\|\rho_h^{u^\star}(t^*_h)\|_{\mathcal T_h}\le C h^{3-d/2}  \quad  \mbox{for all $ h \leq h^*_1 $.}
%	\end{align*}
%	Since $ C $ does not depend on $ h $, there exists $ h^*_2 \leq h^*_1 $ such that $	\|\rho_h^{u^\star}(t^*_h)\|_{L^{\infty}(\Omega)}<1/2$ for all $ h \leq h^*_2 $.  This contradicts with the assumption \eqref{assum}, and therefore, $t^*_h = T$ for all $ h $ small enough.
%\end{proof}

\begin{theorem}\label{main_err_qu2}
	If the nonlinearity $F$ satisfies the local Lipschitz condition in \Cref{Locally_lip}, the mesh is quasi-uniform, $ k \geq 1 $, and the assumptions at the beginning of \Cref{Error_analysis} hold, then for all $ h $ small enough the conclusions of \Cref{main_err_qu} and \Cref{res_coll} are true for all $0\le t\le T$.
\end{theorem}

\section{Numerical Results}
\label{sec:numerics}

In this section, we present two examples to demonstrate the performance of the Interpolatory HDG$_k$  method.

\begin{example}[The Allen-Cahn or Chaffee-Infante equation]
	We begin with an example with an exact solution in order to illustrate the convergence theory.  The domain is the unit square $\Omega = [0,1]\times [0,1]\subset \mathbb R^2$, the nonlinear term is $F( u) = u^3-u$, and the source term $f$ is chosen so that the exact solution is $u = \sin(t)\sin(\pi x)\sin(\pi y)$.  Backward Euler and Crank-Nicolson are applied for the time discretization when $k=0$ and $k=1$, respectively, where $k$ is the degree of the polynomial.  The time step is chosen as $\Delta t = h$ when $k=0$ and $\Delta t   = h^2$ when $k=1$.  We report the errors at the final time $ T = 1 $ for polynomial degrees $ k = 0 $ and $ k = 1 $ in \Cref{table_1}.  The observed convergence rates match the theory.
	\begin{table}%[H]
		\small
		\caption{History of convergence.}\label{table_1}
		\centering
		%%%%%%%%%%%%%%%%%%%%%%%%%%%%%%%55
		Example 1: Errors for $\bm{q}_h$, $u_h$ and $u_h^\star$ of Interpolatory HDG$_k$ {
			\begin{tabular}{c|c|c|c|c|c|c|c}
				\Xhline{1pt}

				\multirow{2}{*}{Degree}
				&\multirow{2}{*}{$\frac{h}{\sqrt{2}}$}	
				&\multicolumn{2}{c|}{$\|\bm{q}-\bm{q}_h\|_{0,\Omega}$}	
				&\multicolumn{2}{c|}{$\|u-u_h\|_{0,\Omega}$}	
				&\multicolumn{2}{c}{$\|u-u_h^\star\|_{0,\Omega}$}	\\
				\cline{3-8}
				& &Error &Rate
				&Error &Rate
				&Error &Rate
				\\
				\cline{1-8}
				\multirow{5}{*}{ $k=0$}
				&$2^{-1}$	&1.2889	&	    &5.0344E-01	&	    &4.5836E-01	&\\
				&$2^{-2}$	&7.0471E-01	&0.87	&2.8491E-01	&0.82 	&2.5673E-01	&0.84\\
				&$2^{-3}$	& 3.5473E-01	&0.99	&1.5511E-01	&0.88 	&1.4105E-01	&0.86\\
				&$2^{-4}$	&1.7648E-01	&1.00	&8.0617E-02	&0.94 	&7.3725E-02	&0.94\\
				&$2^{-5}$	&8.7855E-02	&1.00 	&4.1025E-02	&0.97 	& 3.7627E-02	&0.97\\

				\cline{1-8}
				\multirow{5}{*}{$k=1$}
				
				%%&$2^2\times2^2$	&5.4821E-01	&	    &5.9035E-01	&	    &5.9046E-01	&\\
				%%&$2^3\times2^3$	&2.7265E-01	&1.01 	&2.9041E-01	&1.02 	&2.9046E-01	&1.02\\
				&$2^{-1}$	&	3.7304E-01	& 	       &1.7028E-01	& 	      &3.0236E-02	    &\\
				&$2^{-2}$	&9.9820E-02	&1.90 	&4.8288E-02	&1.82 	&3.9074E-03	&2.95\\
				&$2^{-3}$	&2.5307E-02	&1.98 	&1.2561E-02	&1.94 	&4.7940E-04	&3.02\\
				&$2^{-4}$	&6.3422E-03	&2.00 	&3.1825E-03	&1.98	&5.9047E-05	&3.02\\
				&$2^{-5}$	&1.5858E-03	&2.00 	&7.9966E-04	&2.00 	&7.3168E-06	&3.01\\

				\Xhline{1pt}

			\end{tabular}
		}
	\end{table}
\end{example}

\begin{example}[The Schnakenberg model]
Next, we consider a more complicated example of a reaction diffusion PDE system with zero Neumann boundary conditions that does not satisfy the assumptions for the convergence theory established here.  We consider such an example to demonstrate the applicability of the Interpolatory HDG$_k$ method to more general problems.  

Specifically, we consider the Schnakenberg model, which has been used to model the spatial distribution of a morphogen; see \cite{MR2511741} for more details. The Schnakenberg system has the form
\begin{align*}
\frac{\partial C_a}{\partial t} & = D_1\nabla^2 C_a + \kappa(a - C_a  +C_a^2 C_i),\\
\frac{\partial C_i}{\partial t} & = D_2\nabla^2 C_i+ \kappa(b - C_a^2 C_i),
\end{align*}
with initial conditions
\begin{align*}
C_a(\cdot,0) & =a+b + 10^{-3}\exp\big(-100((x-1/3)^2 + (y-1/2)^2)\big),\\
C_i(\cdot,0)& = \frac {b}{(a+b)^2},
\end{align*}
and homogeneous Neumann boundary conditions. The parameter values are $\kappa = 100$, $a = 0.1305$, $b = 0.7695$, $D_1 = 0.05$, and $D_2 = 1$. We choose polynomial degree $k=1$ and  apply Crank-Nicolson for the time discretization with time step $\Delta t = 0.001$.

We vary the spatial domain, but keep all of parameters in the model unchanged. The first domain is the unit square $\Omega = [0,1]\times [0,1]$, and the domain is partitioned into $2048$ elements.  The second domain is the circle $\Omega = \{(x,y) : (x-0.5)^2 + (y - 0.5)^2 <0.5^2\}$ and we use $7168$ elements.  %The third domain is the narrow rectangle $ \Omega= [0, 0.05] \times [0, 1]$ and we use $2048$ elements.

Numerical results are shown in \Cref{fig:rd_system_square}--\Cref{fig:rd_system_circle}.  Spot patterns form on the square and circular domains.  Our numerical results are very similar to results reported in \cite{MR2511741}.
%
%It is interesting to notice that spot-like patterns are formed on the unite square and  circular domains, see \Cref{squre,circle}, but stripe-like patterns are formed on the narrow rectangular domain \Cref{strip}. These simulations of the Schnakenberg system on different domains provide an example of the sensitivity of patterns in reaction-diffusion systems with respect to the domain size and shape. 
\begin{figure}
	\centerline{
		\hbox{\includegraphics[height=2.5in]{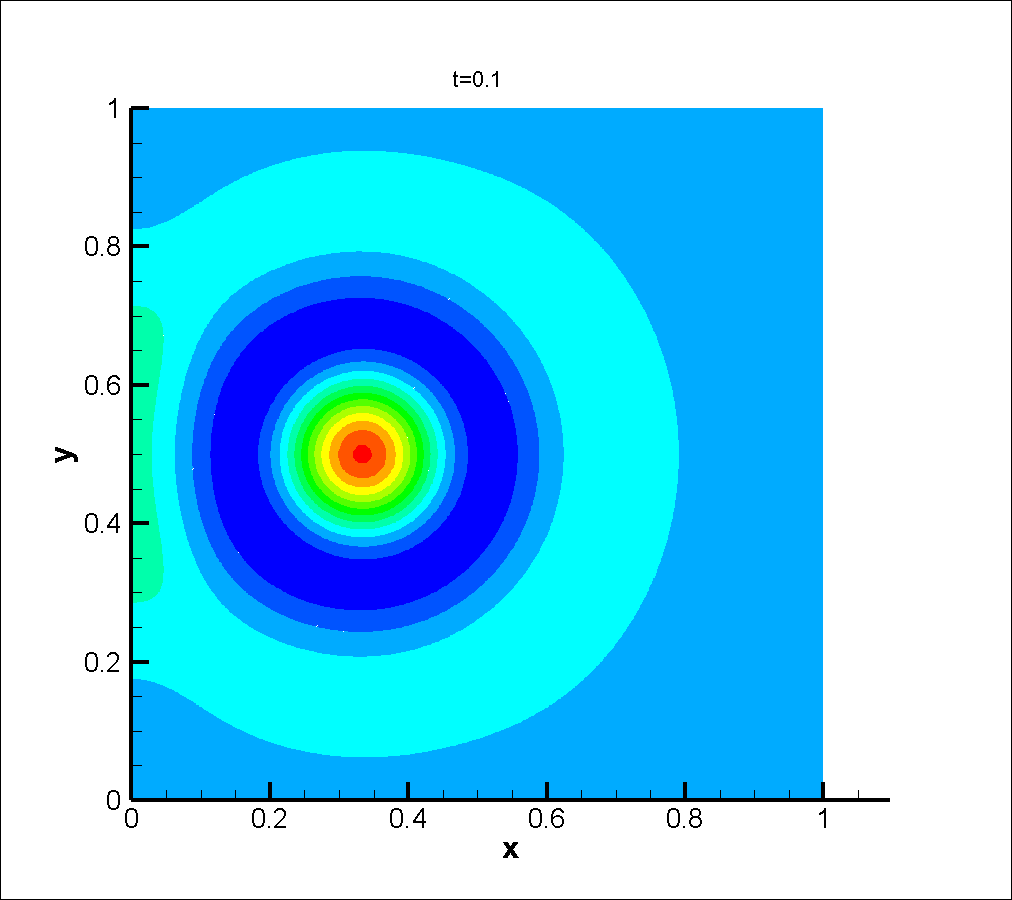}}
		\hbox{\includegraphics[height=2.5in]{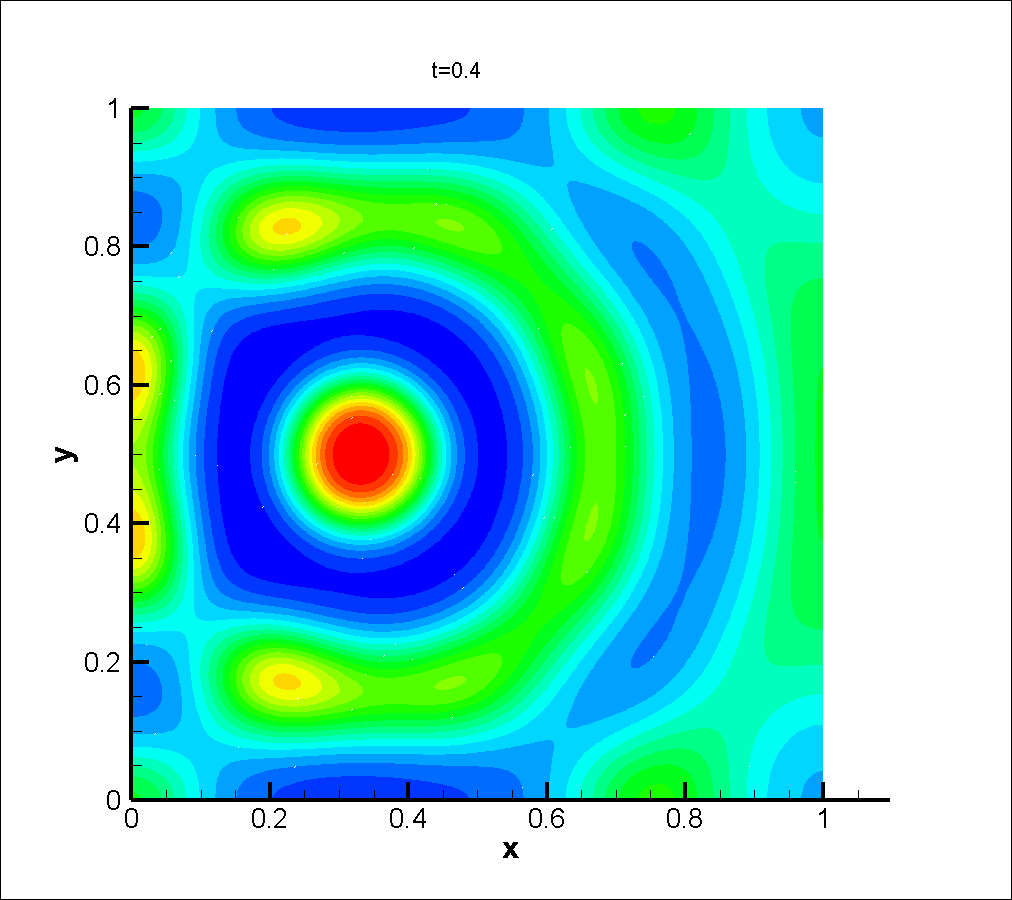}}
	}
	\centerline{
	\hbox{\includegraphics[height=2.5in]{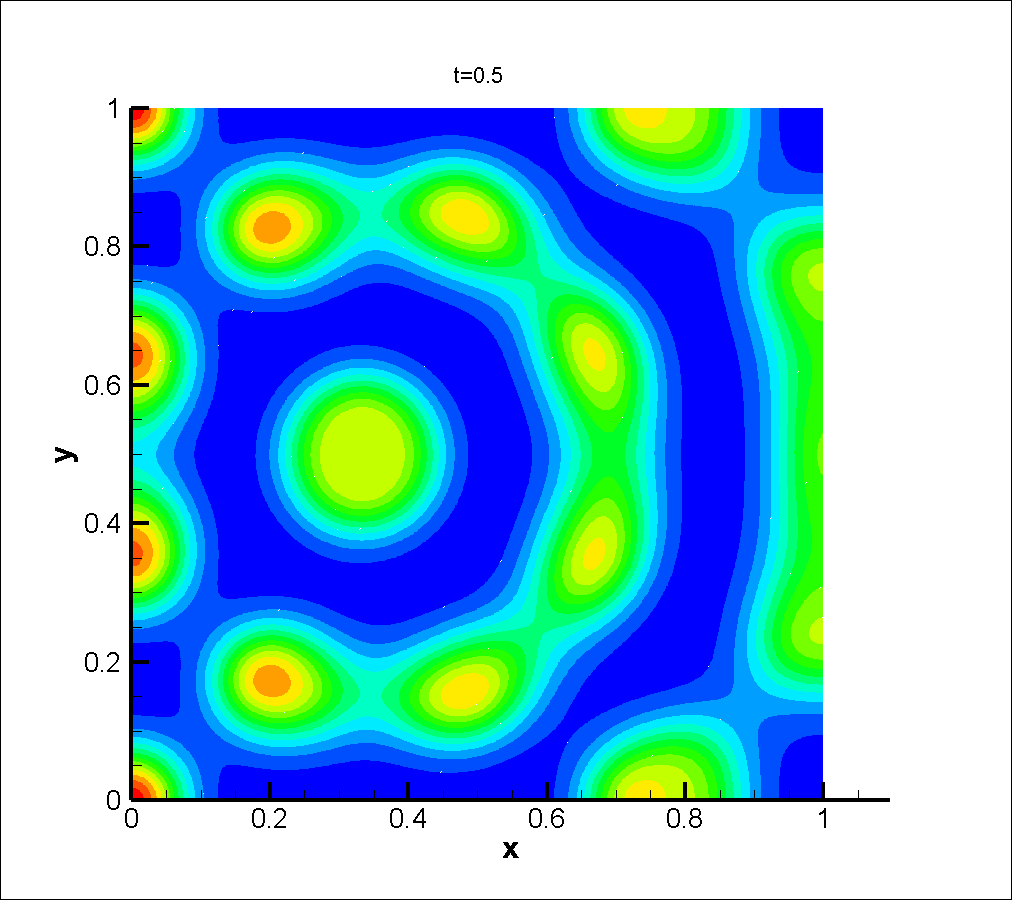}}
	\hbox{\includegraphics[height=2.5in]{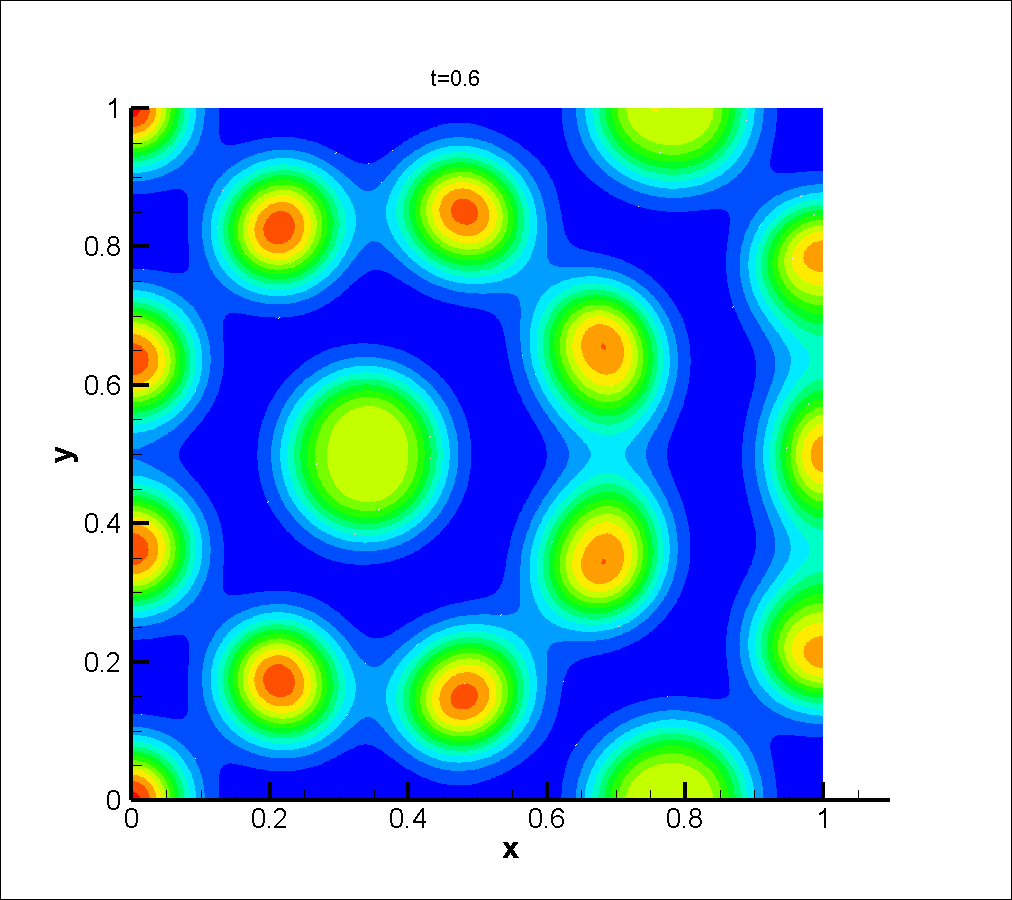}}
}
	\centerline{
	\hbox{\includegraphics[height=2.5in]{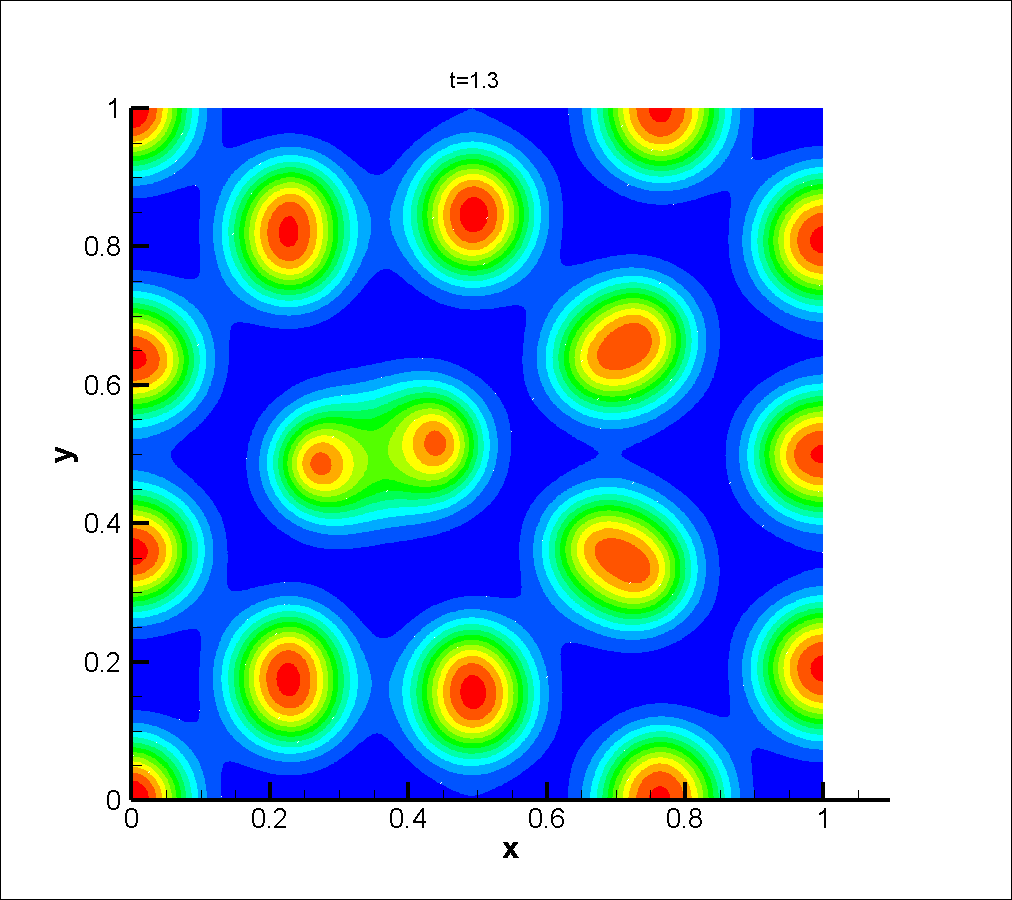}}
	\hbox{\includegraphics[height=2.5in]{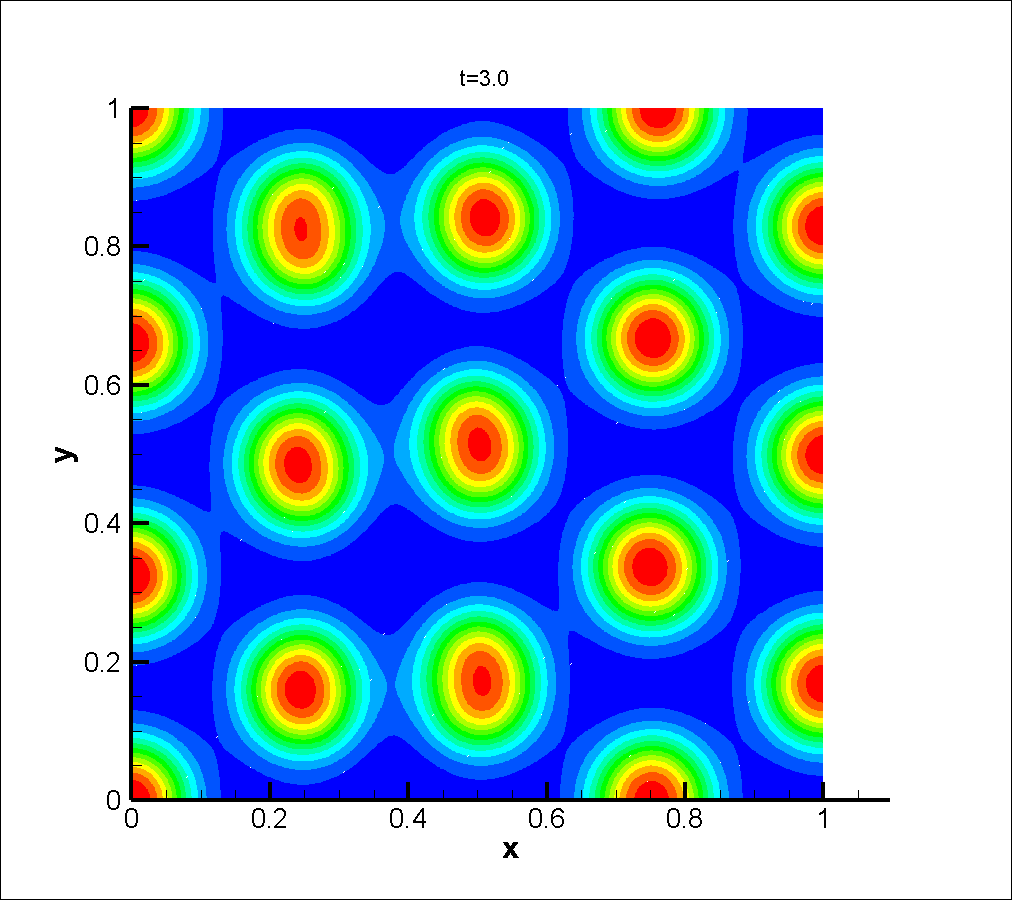}}
}

	\caption{\label{fig:rd_system_square} Contour plots of the time evolution of the concentration of the activator $C_a$ on the unit square.}
	\label{squre}
	\centering
\end{figure}

\begin{figure}
	\centerline{
		\hbox{\includegraphics[height=2.5in]{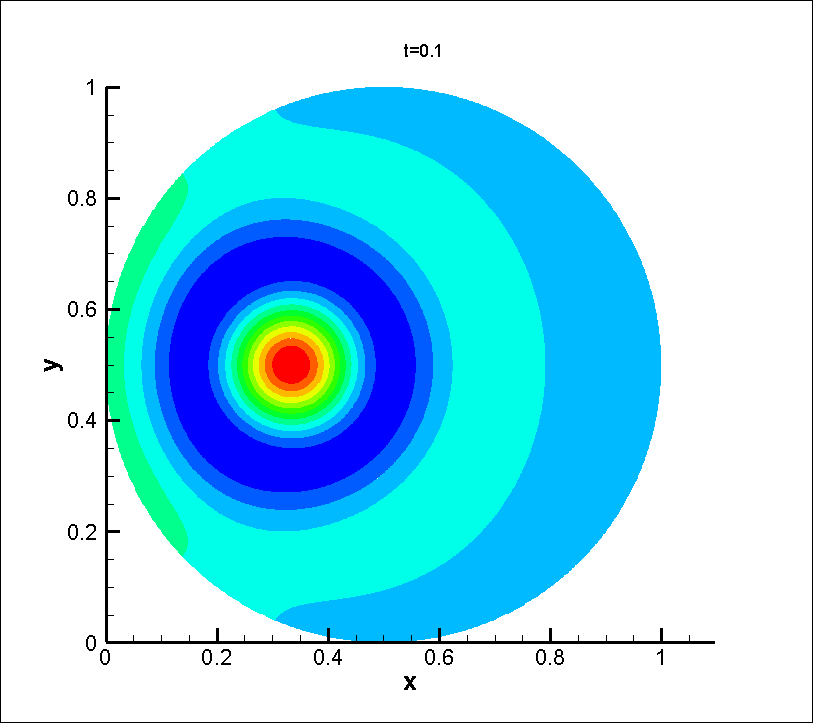}}
		\hbox{\includegraphics[height=2.5in]{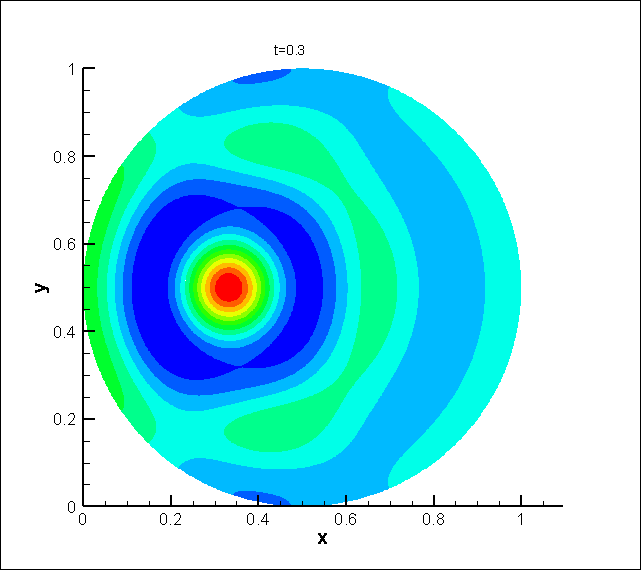}}
		%\hbox{\includegraphics[height=1.2in]{circle_v04.png}}
	}
	\centerline{
		\hbox{\includegraphics[height=2.5in]{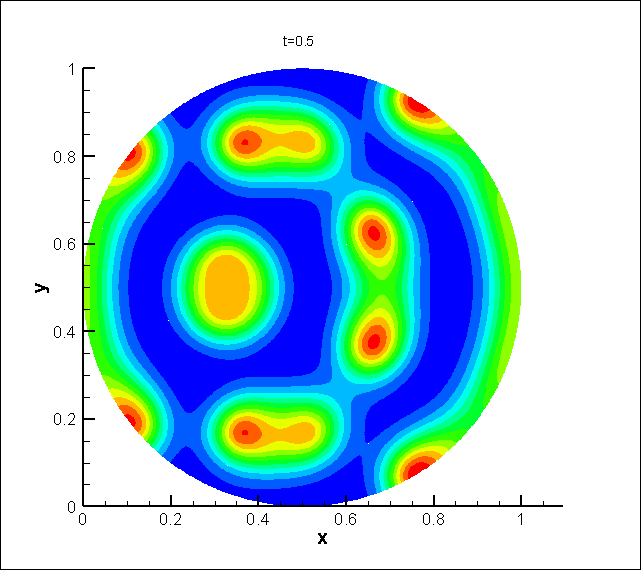}}
     	\hbox{\includegraphics[height=2.5in]{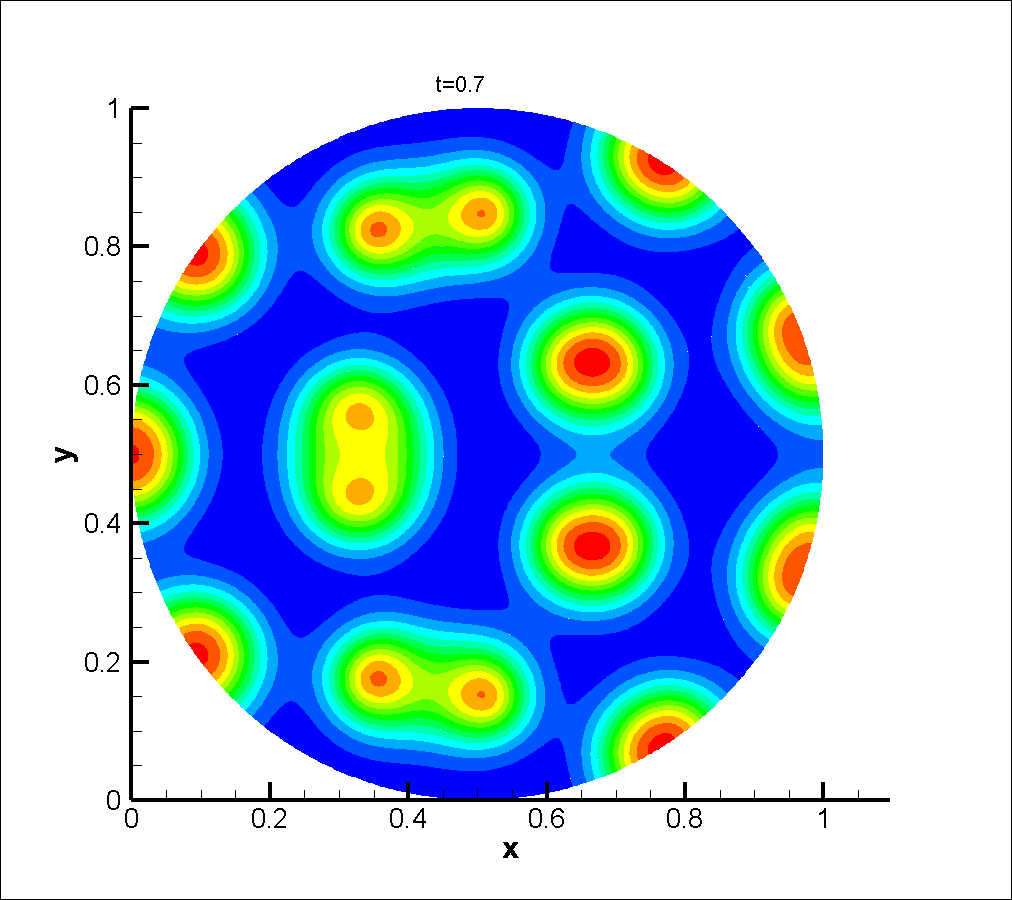}}
%    	\hbox{\includegraphics[height=1.2in]{circle_v08.png}}
%    	\hbox{\includegraphics[height=1.2in]{circle_v12.png}}
	}
	\centerline{
	}

	\centerline{
	\hbox{\includegraphics[height=2.5in]{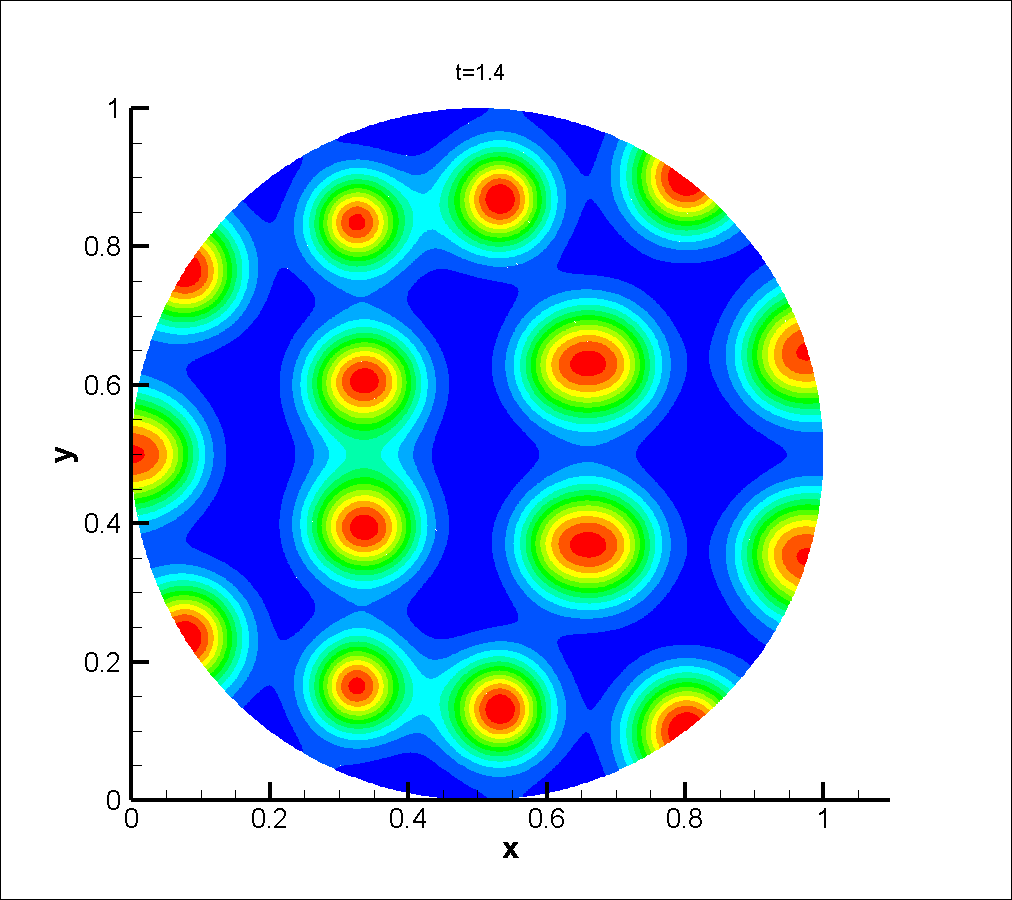}}
%	\hbox{\includegraphics[height=1.2in]{circle_v15.png}}
%	\hbox{\includegraphics[height=1.2in]{circle_v20.png}}
	\hbox{\includegraphics[height=2.5in]{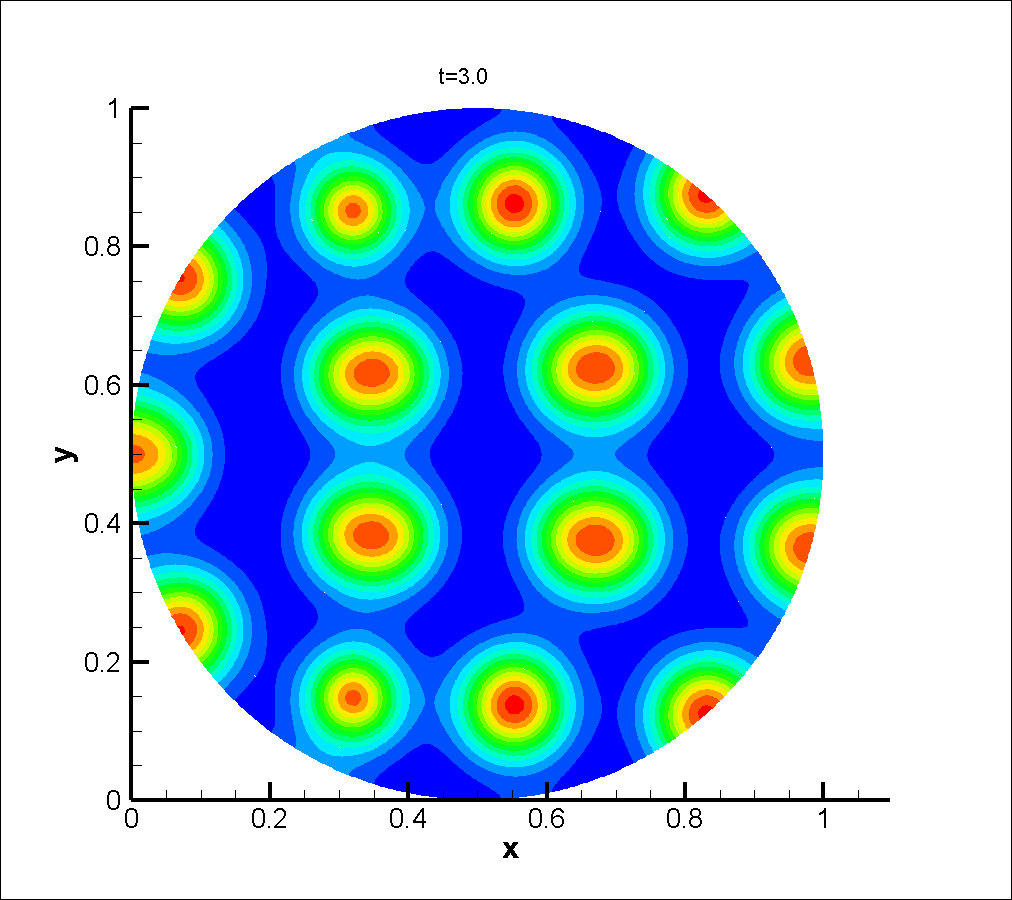}}
}

		\caption{\label{fig:rd_system_circle} Contour plots of the time evolution of the concentration of the activator $C_a$ on the circular domain.}
	\label{circle}
	\centering
\end{figure}

%
%\begin{figure}
%	\centerline{
%		\hbox{\includegraphics[height=2.5in]{strip_v01.png}}
%		\hbox{\includegraphics[height=2.5in]{strip_v02.png}}
%		\hbox{\includegraphics[height=2.5in]{strip_v04.png}}
%		\hbox{\includegraphics[height=2.5in]{strip_v05.png}}
%		\hbox{\includegraphics[height=2.5in]{strip_v08.png}}
%		\hbox{\includegraphics[height=2.5in]{strip_v10.png}}
%		\hbox{\includegraphics[height=2.5in]{strip_v15.png}}
%		\hbox{\includegraphics[height=2.5in]{strip_v20.png}}
%		%\hbox{\includegraphics[height=1.2in]{circle_v04.png}}
%	}
%
%	
%		\caption{\label{fig:rd_system_rectangle} Contour plots of the time evolution of the concentration of the activator $C_a$ on a narrow rectangular domain.}
%	\label{strip}
%	\centering
%\end{figure}

\end{example}

\section{Conclusion}

In our earlier work \cite{CockburnSinglerZhang1}, we considered an Interpolatory HDG$_k$ methods for semilinear parabolic PDEs with a general nonlinearity of the form $F(\nabla u, u)$.  The interpolatory approach achieves optimal convergence rates and reduces the computational cost compared to standard HDG since all of the HDG matrices are assembled once before the time stepping procedure.  However, the method does not have superconvergence by postprocessing.

In this work, we proposed a new superconvergent Interpolatory HDG$_k$  method for approximating the solution of reaction diffusion PDEs. Unlike our earlier Interpolatory HDG$_k$ work \cite{CockburnSinglerZhang1}, the new method uses a postprocessing $u_h^\star$ to evaluate the nonlinear term.  This change provides the superconvergence, and the new method also keeps all of the computational advantages of using an interpolatory approach for the nonlinear term.  We proved the superconvergence under a global Lipschitz condition for the nonlinearity, and then extended the superconvergence results to a local Lipschitz condition assuming the mesh is quasi-uniform.

In the second part of this work \cite{ChenCockburnSinglerZhang2}, we again consider reaction diffusion equations and extend the ideas here to derive other superconvergent interpolatory HDG methods inspired by hybrid high-order methods \cite{MR3507267}.  However, it is currently not clear whether the present approach can be used to obtain the superconvergence for semilinear PDEs with a general nonlinearity $F(\nabla u, u)$.  We are currently exploring this issue.

\section{Appendix A}\label{AppendixA}

Recall the steady state problem \eqref{HDGO_a2} from \Cref{subsec:superconv_duality_argument}, which we repeat here for convenience: let $(\overline{\bm q}_h,\overline{u}_h,\widehat{\overline u}_h)\in \bm V_h\times W_h\times M_h$ be the solution of
\begin{align}\label{HDGO_aa2}
\mathscr B (\overline{\bm q}_h,\overline{u}_h,\widehat{\overline u}_h, \bm r_h, v_h, \widehat v_h )= (f - \Pi_W u_t - F(u),v_h)_{\mathcal{T}_h},
\end{align}
for all $(\bm r_h,v_h,\widehat{v}_h)\in \bm V_h\times W_h\times M_h$.  Since $ \Pi_W $ commutes with the time derivative, taking the partial derivative of \eqref{HDGO_aa2} with respect to $t$ shows $(\partial_t \overline{\bm q}_h,\partial_t\overline{u}_h,\partial_t\widehat{\overline u}_h)\in \bm V_h\times W_h\times M_h$ is the solution of
\begin{align}\label{HDGO_3}
\mathscr B (\partial_t\overline{\bm q}_h,\partial_t\overline{u}_h,\partial_t\widehat{\overline u}_h, \bm r_h, v_h, \widehat v_h )&= (f_t - \Pi_W u_{tt} - F'(u)u_t,v_h)_{\mathcal{T}_h},
\end{align}
for all $(\bm r_h,v_h,\widehat{v}_h)\in \bm V_h\times W_h\times M_h$.

The proof of the following lemma is very similar to a proof in \cite{MR2629996}, hence we omit it here. 
\begin{lemma}\label{error_u1_appendix}
	For $\varepsilon_h^{\overline{\bm q}}=\bm{\Pi}_{V}\bm q- \overline{\bm q}_h $, $ \varepsilon_h^{ \overline u}=\Pi_{W} {u}-\overline u_h $, and $ \varepsilon_h^{ \widehat{\overline u}}=P_M u-\widehat{\overline u}_h$, we have
	\begin{align}\label{error_u11appendix}
	\mathscr B(\varepsilon_h^{\overline{\bm q}},\varepsilon_h^{\overline u},\varepsilon_h^{\widehat {\overline u}}; \bm r_h, w_h, \widehat v_h) =  (\bm \Pi_V \bm{q} - \bm q,\bm{r}_h)_{\mathcal{T}_h}+ (\Pi_W u_t - u_t , v_h)_{\mathcal{T}_h},
	\end{align}
	for all $(\bm r_h,v_h,\widehat v_h)\in \bm V_h\times W_h\times M_h$.
\end{lemma}

The next step is the consideration of the dual problem \eqref{Dual_PDE1_assumption}, which we again repeat for convenience:  Let 
\begin{equation}\label{Dual_PDE1}
\begin{split}
\bm{\Phi}+\nabla\Psi&=0\qquad\qquad\text{in}\ \Omega,\\
\nabla\cdot\bm \Phi   &=\Theta\qquad\quad~~\text{in}\ \Omega,\\
\Psi &= 0\qquad\qquad\text{on}\ \partial\Omega.
\end{split}
\end{equation}
By the assumption at the beginning of \Cref{Error_analysis}, this boundary value problem admits the regularity estimate
\begin{align}\label{regularity_PDE}
\|\bm \Phi\|_{H^{1}(\Omega)} + \|\Psi\|_{H^{2}(\Omega)} \le C   \|\Theta\|_{L^{2}(\Omega)},
\end{align}
for all $\Theta \in L^2(\Omega)$.

\begin{lemma}\label{dual_ar}
	We have
	\begin{align*}
	\|\varepsilon^{\overline u}_h\|_{\mathcal{T}_h} &\le  Ch^{\min\{k,1\}} (\|\bm q - \bm\Pi_V \bm q\|_{\mathcal T_h}+\|u_t - \Pi_W u_t\|_{\mathcal T_h})\\
	\|\varepsilon^{\overline {\bm q}}_h\|_{\mathcal{T}_h} &\le  C (\|\bm q - \bm{\Pi}_V \bm q\|_{\mathcal T_h} +\|u_t - \Pi_W u_t\|_{\mathcal T_h}).
	\end{align*}
\end{lemma}
\begin{proof}
	Let $\Theta = \varepsilon_h^{\overline u}$ in the dual problem \eqref{Dual_PDE1}, and take $(\bm r_h,v_h,\widehat v_h) = (-\bm\Pi_V\bm{\Phi}, \Pi_W\Psi,P_M \Psi)$ in the definition of $ \mathscr{B} $ \eqref{def_B} to get
	\begin{align*}
	\mathscr B &(\varepsilon^{\overline{\bm q}}_h,\varepsilon^{\overline u}_h,\varepsilon^{\widehat {\overline u}}_h;-\bm\Pi_V\bm{\Phi},\Pi_W\Psi,P_M\Psi)\\
	&=-(\varepsilon^{\overline{\bm q}}_h,\bm\Pi_V\bm{\Phi})_{\mathcal{T}_h}+(\varepsilon^{\overline u}_h,\nabla\cdot\bm\Pi_V\bm{\Phi})_{\mathcal{T}_h}-\langle \varepsilon^{\widehat {\overline u}}_h, \bm\Pi_V\bm{\Phi}\cdot \bm{n}\rangle_{\partial\mathcal{T}_h}+ (\nabla\cdot \varepsilon^{\overline{\bm q}}_h,  \Pi_W\Psi)_{\mathcal{T}_h}\\
	&\quad -\langle \varepsilon^{\overline{\bm q}}_h\cdot \bm{n}, P_M \Psi \rangle_{\partial {\mathcal{T}_h}} + \left\langle\tau(\varepsilon^{\overline u}_h-\varepsilon^{\widehat {\overline u}}_h), \Pi_W \Psi - P_M\Psi \right\rangle_{\partial\mathcal{T}_h}\\
	&=-(\varepsilon^{\overline {\bm q}}_h,\bm{\Phi})_{\mathcal{T}_h}+(\varepsilon^{\overline{\bm q}}_h,\bm \Phi - \bm \Pi_V \bm\Phi)_{\mathcal{T}_h}+(\varepsilon^{\overline u}_h,\nabla\cdot\bm{\Phi})_{\mathcal{T}_h}-(\varepsilon^{\overline u}_h,\nabla\cdot (\bm \Phi - \bm \Pi_V \bm\Phi))_{\mathcal{T}_h}\\
	&\quad +\langle \varepsilon^{\widehat {\overline u}}_h,  (\bm \Phi - \bm \Pi_V \bm\Phi)\cdot \bm{n}\rangle_{\partial\mathcal{T}_h} + (\nabla\cdot \varepsilon^{\overline{\bm q}}_h, \Psi)_{\mathcal{T}_h} + (\nabla\cdot \varepsilon^{\overline{\bm q}}_h, \Pi_W \Psi - \Psi)_{\mathcal{T}_h}\\
	&\quad -\langle \varepsilon^{\overline{\bm q}}_h\cdot \bm{n},  \Psi \rangle_{\partial {\mathcal{T}_h}} + \left\langle\tau(\varepsilon^{\overline u}_h-\varepsilon^{\widehat {\overline u}}_h), \Pi_W \Psi - P_M\Psi \right\rangle_{\partial\mathcal{T}_h}\\
	&=(\varepsilon^{\overline{\bm q}}_h,\bm \Phi - \bm \Pi_V \bm\Phi)_{\mathcal{T}_h}  + \|\varepsilon^{\overline u}_h\|_{\mathcal T_h}^2.
	\end{align*}
	On the other hand, take $(\bm r_h,v_h,\widehat v_h) = (-\bm\Pi_V\bm{\Phi}, \Pi_W\Psi,P_M \Psi)$ in \eqref{error_u11appendix} to get
	\begin{align}\label{two_e}
	\mathscr B (\varepsilon^{\overline{\bm q}}_h,\varepsilon^{\overline u}_h,\varepsilon^{\widehat {\overline u}}_h;-\bm\Pi_V\bm{\Phi},\Pi_W\Psi,P_M\Psi)=  (\bm q - \bm\Pi_V \bm q,\bm\Pi_V\bm{\Phi})_{\mathcal T_h} +  (\Pi_W u_t - u_t , \Pi_W \Psi)_{\mathcal{T}_h}.
	\end{align}
	Comparing the above two equalities gives
	\begin{align*}
	\|\varepsilon^{\overline u}_h\|^2_{\mathcal{T}_h}
	&=-(\varepsilon^{\overline{\bm q}}_h,\bm \Phi - \bm \Pi_V \bm\Phi)_{\mathcal{T}_h}+ (\bm q - \bm{\Pi}_V \bm q,\bm\Pi_V\bm{\Phi})_{\mathcal T_h} +  (\Pi_W u_t - u_t , \Pi_W \Psi)_{\mathcal{T}_h} \\
	&=-(\varepsilon^{\overline{\bm q}}_h,\bm \Phi - \bm \Pi_V \bm\Phi)_{\mathcal{T}_h}+ (\bm q - \bm\Pi_V \bm q,\bm\Pi_V\bm{\Phi} - \bm \Phi)_{\mathcal T_h}  \\
	&\quad +  (\bm q - \bm\Pi_V \bm q,\bm \Phi)_{\mathcal T_h}  +  (\Pi_W u_t - u_t , \Pi_W \Psi)_{\mathcal{T}_h} \\
	&=-(\varepsilon^{\overline{\bm q}}_h,\bm \Phi - \bm \Pi_V \bm\Phi)_{\mathcal{T}_h}+ (\bm q - \bm\Pi_V \bm q,\bm\Pi_V\bm{\Phi} - \bm \Phi)_{\mathcal T_h}  \\
	&\quad -  (\bm q - \bm\Pi_V \bm q,\nabla \Psi)_{\mathcal T_h} +  (\Pi_W u_t - u_t , \Pi_W \Psi)_{\mathcal{T}_h} \\
	&=-(\varepsilon^{\overline{\bm q}}_h,\bm \Phi - \bm \Pi_V \bm\Phi)_{\mathcal{T}_h}+ (\bm q - \bm\Pi_V \bm q,\bm\Pi_V\bm{\Phi} - \bm \Phi)_{\mathcal T_h}  \\
	&\quad -  (\bm q - \bm\Pi_V \bm q,\nabla (\Psi - \Pi_W\Psi))_{\mathcal T_h} +  (\Pi_W u_t - u_t , \Pi_W \Psi - \min\{k,1\}\Pi_0 \Psi)_{\mathcal{T}_h}.
	\end{align*}
	Hence, by the regularity of the dual PDE \eqref{regularity_PDE}, we have 
	\begin{align}\label{errir_Rew}
	\|\varepsilon^{\overline u}_h\|^2_{\mathcal{T}_h} \le Ch^2 \|\varepsilon^{\overline{\bm q}}_h\|_{\mathcal T_h}^2 + Ch^{\min\{2k,2\}}\|\bm q - \bm{\Pi}_V \bm q\|_{\mathcal T_h}^2 + Ch^{\min\{2k,2\}}\|u_t - \Pi_W u_t\|_{\mathcal T_h}^2.
	\end{align}
	
	Next, take $(\bm r_h,v_h,\widehat{v}_h)=(\varepsilon_h^{\overline {\bm q}},\varepsilon_h^{\overline u},\varepsilon_h^{\widehat {\overline u}})$ in \eqref{error_u11appendix} to obtain
	\begin{align*}
	\hspace{1em}&\hspace{-1em}	\|\varepsilon_h^{\overline{\bm q}}\|^2_{\mathcal{T}_h}+\langle \tau(\varepsilon_h^{\overline u} -\varepsilon_h^{\widehat{ \overline u}}),  \varepsilon_h^{\overline u} -\varepsilon_h^{\widehat{ \overline u}} \rangle_{\partial{\mathcal{T}_h}}  \\
	&= (\bm\Pi_V {\bm{q}} -\bm q, \varepsilon_h^{\overline {\bm q}})_{\mathcal{T}_h} + (\Pi_W u_t - u_t,\varepsilon_h^{\overline u})_{\mathcal T_h}\\
	&\le C\|\bm\Pi_V {\bm{q}} -\bm q\|_{\mathcal T_h}^2 + \frac 1 2 	\|\varepsilon_h^{\overline {\bm q}}\|^2_{\mathcal{T}_h} + 4C\| \Pi_W u_t - u_t\|_{\mathcal T_h}^2 +\frac 1 {4C} \|\varepsilon_h^{\overline u}\|_{\mathcal T_h}^2.
	\end{align*}
	This implies
	\begin{align}\label{energy_q_appdenix}
	\|\varepsilon_h^{\overline{\bm q}}\|^2_{\mathcal{T}_h}+\langle \tau(\varepsilon_h^{\overline u} -\varepsilon_h^{\widehat{ \overline u}}),  \varepsilon_h^{\overline u} -\varepsilon_h^{\widehat{ \overline u}} \rangle_{\partial{\mathcal{T}_h}} 
	\le 2C\|\bm\Pi_V {\bm{q}} -\bm q\|_{\mathcal T_h}^2 + 8C\| \Pi_W u_t - u_t\|_{\mathcal T_h}^2 +\frac 1 {2C} \|\varepsilon_h^{\overline u}\|_{\mathcal T_h}^2.
	\end{align}
	Next, use $h\le 1$ and substitute \eqref{energy_q_appdenix} into \eqref{errir_Rew} to yield the result.
%	\begin{align*}
%	\|\varepsilon^{\overline u}_h\|_{\mathcal{T}_h} &\le  Ch^{\min\{k,1\}} (\|\bm q - \bm{\Pi}_V \bm q\|_{\mathcal T_h}+\|u_t - \Pi_W u_t\|_{\mathcal T_h}),\\
%	\|\varepsilon^{\overline {\bm q}}_h\|_{\mathcal{T}_h} &\le  C (\|\bm q - \bm{\Pi}_V \bm q\|_{\mathcal T_h} +\|u_t - \Pi_W u_t\|_{\mathcal T_h}).
%	\end{align*}
\end{proof}

Following the same steps, we obtain the following result:
\begin{lemma}\label{dual_ar2}
	We have
	\begin{align*}
	\|\partial_t(\Pi_W u - \overline u_h)\|_{\mathcal{T}_h} &\le  Ch^{\min\{k,1\}} (\|\bm q_t - \bm\Pi_V \bm q_t\|_{\mathcal T_h} +\|u_{tt} - \Pi_W u_{tt}\|_{\mathcal T_h}).
	\end{align*}
\end{lemma}

\section*{Acknowledgements}
G.\ Chen thanks Missouri University of Science and Technology for hosting him as a visiting scholar; some of this work was completed during his research visit. J.\ Singler and Y.\ Zhang were supported in part by National Science Foundation grant DMS-1217122. J.\ Singler and Y.\ Zhang thank the IMA for funding research visits, during which some of this work was completed.

\bibliographystyle{plain}
\bibliography{yangwen_ref_papers}

\end{document}